\documentclass[11pt,twoside]{amsart}
\usepackage{a4wide}
\usepackage{url}
\usepackage{ifthen}
\usepackage{amssymb}
\usepackage{amscd}
\usepackage{amsmath}
\usepackage[all]{xy}
\CompileMatrices
\usepackage[backref]{hyperref}
\usepackage[usenames]{color}

 \newtheorem{thm}{Theorem}[section]
 \newtheorem{prop}[thm]{Proposition}
 \newtheorem{cor}[thm]{Corollary}
 \newtheorem{lem}[thm]{Lemma}

\theoremstyle{definition}

\theoremstyle{remark}
\newtheorem{rem}[thm]{Remark}

\numberwithin{equation}{section}

\newcommand{\To}{\longrightarrow}

  1

\renewcommand{\L}{\ifmmode {\mathcal{L}}\else$\mathcal{L}$\ \fi}

\newcommand{\id}{\mathrm{id}}
\newcommand{\bbC}{\ifmmode {\mathbb{C}}\else$\mathbb{C}$\ \fi}
\newcommand{\bbR}{\ifmmode {\mathbb{R}}\else$\mathbb{R}$\ \fi}

\newcommand{\be}{\begin{equation}}
\newcommand{\ee}{\end{equation}}

\newcommand{\fpbar}{\ifmmode {\overline{\mathbb{F}_p}}\else$\mathbb{F}_p$\ \fi}

\newcommand{\fp}{\ifmmode {\mathbb{F}_p}\else$\mathbb{F}_p$\ \fi}
\newcommand{\zp}{\ifmmode {\mathbb{Z}_p}\else$\mathbb{Z}_p$\ \fi}
\newcommand{\Zp}{\ifmmode {\mathbb{Z}_p}\else$\mathbb{Z}_p$\ \fi}
 \newcommand{\z}{\mathbb{Z}}

\newcommand{\zpMod}{\ifmmode\mbox{$\zp$-Mod}\else$\zp$-Mod \fi}
\newcommand{\Mod}{\ifmmode\mbox{$\Lambda$-Mod}\else$\Lambda$-Mod \fi}
\renewcommand{\mod}{\ifmmode\mbox{$\Lambda$-mod}\else$\Lambda$-mod
\fi}

\newcommand{\La}{\ifmmode\Lambda\else$\Lambda$\fi}

\newcommand{\End}{{\mathrm{End}}}

\newcommand{\Jac}{{\mathrm{Jac}}}

\newcommand{\M}{\ifmmode {\mathfrak M}\else${\mathfrak M}$ \fi}
\newcommand{\m}{\ifmmode {\mathfrak m}\else$\mathfrak m$ \fi}
\newcommand{\mh}{\ifmmode {\mathfrak m}(H)\else${\mathfrak m}(H)$ \fi}

\newcommand{\p}{\ifmmode {\mathfrak p}\else${\mathfrak p}$\ \fi}

\renewcommand{\P}{\ifmmode {\mathfrak P}\else${\mathfrak P}$\ \fi}

\newcommand{\e}{\ifmmode {\mathcal{E}}\else$\mathcal{E}$ \fi}

\newcommand{\n}{\mbox{${\mathbb N}$}}

\newcommand{\C}{\mathcal{C}}
\newcommand{\N}{\mathcal{N}}

\renewcommand{\O}{\mathcal{ O}}

\newcommand{\G}{\ifmmode {\mathcal{G}}\else${\mathcal{G}}$\ \fi}

\newcommand{\A}{\ifmmode {\mathcal{A}}\else${\mathcal{ A}}$\ \fi}

\renewcommand{\projlim}[1] {{\lim\limits_{\stackrel{\displaystyle
\longleftarrow}{#1}}}}

\newcommand{\lk}{\hspace{-3 pt}\ll\hspace{-2 pt}}

\newcommand{\Qp}{\ifmmode {{\mathbb Q}_p}\else${\mathbb Q}_p$\ \fi}
\newcommand{\qp}{\ifmmode {{\mathbb Q}_p}\else${\mathbb Q}_p$\ \fi}
\newcommand{\ql}{\ifmmode {{\mathbb Q}_l}\else${\mathbb Q}_l$\ \fi}
\newcommand{\Q}{\ifmmode {\mathbb Q}\else${\mathbb Q}$\ \fi}
\newcommand{\q}{\ifmmode {\mathbb Q}\else${\mathbb Q}$\ \fi}

\newcommand{\gr}{\mathrm{gr}}

\begin{document}

\title[]{Localisations and completions of skew power series rings}%
\author{Peter Schneider and Otmar Venjakob}%
\address{Universit\"{a}t M\"{u}nster,  Mathematisches Institut,  Einsteinstr. 62,
48291 M\"{u}nster,  Germany,
 http://www.uni-muenster.de/math/u/schneider/ }%
\email{pschnei@math.uni-muenster.de }%

\address{Universit\"{a}t Heidelberg,  Mathematisches Institut,  Im Neuenheimer Feld 288,  69120
Heidelberg,  Germany,
 http://www.mathi.uni-heidelberg.de/$\,\tilde{}\,$venjakob/}
\email{venjakob@mathi.uni-heidelberg.de}

\thanks{ }%

\date{\today}%
\begin{abstract}
This paper is a natural continuation of the study of skew power series rings
$A=R[[t;\sigma,\delta]]$ initiated in \cite{sch-ven}. We construct skew Laurent series rings $B$
and show the existence of some canonical Ore sets $S$ for the skew power series rings $A$ such that
a certain completion of the localisation $A_S$  is isomorphic to $B.$ This is applied to certain
Iwasawa algebras. Finally we introduce subrings of overconvergent skew Laurent series rings.
\end{abstract}
\maketitle

\section*{Introduction}

In  \cite{sch-ven} we introduced the general notion of a skew power series ring
$A=R[[t;\sigma,\delta]]$ over a pseudocompact ring $R$ (see section 1 for details) and we studied
basic properties of it. This was motivated by and applied to the study of Iwasawa algebras of certain $p$-adic Lie groups. It is a very important step in non-commutative Iwasawa theory to pass from the Iwasawa algebra to its localisation in a specific rather big Ore subset (cf.\ \cite{cfksv}). The starting point of the present paper is the discovery that a natural completion of this localisation can be viewed as a skew Laurent series ring. This means in particular that this completed localisation can be constructed by inverting a single element $t$ with subsequent completion.\\

In fact we will develop our results in the context of a general skew power series ring $A=R[[t;\sigma,\delta]]$. In section 1 we want to formally invert the element $t,$ i.e.\ construct a skew
Laurent series ring
\[ B=R((t;\sigma,\delta)).\] It turns out that for a (left of right) Artinian ring $R$ the ring $B$
exists in form of the localisation $A_T$ of $A$ with respect to $T=\{1,t,t^2,\ldots \},$ see
subsection 1.1. But for general $R$ and non-trivial $\delta$ one sees easily that a ring extension
of $A$ in which  $t $ is invertible in general contains  series $\sum_{i\in\z}r_iT^i$ with {\em infinite} negative part. Due to this observation we  define in subsection 1.2 a, in both
directions, infinite skew Laurent series ring
\[B=R\hspace{-3 pt}\ll\hspace{-2 pt}t;\sigma,\delta]]\] consisting of formal infinite sums $\sum_{i\in\z}r_iT^i$ such that the
negative part satisfies some growth condition; in fact the commutative version, i.e.\ for
trivial $\sigma$ and $\delta,$ with for example $R=\zp$ is a well-known ring in $p$-adic analysis.
The multiplication law is given by explicit formulae \eqref{mult}. In order to see that they
actually define a (topological) ring structure we identify $B$ with the projective limit of skew
Laurent series rings with Artinian coefficients
\[   B\cong \projlim{k}\big(R/I_k\big)((t;\overline{\sigma},\overline{\delta})) \]
for a certain filtration $I_\bullet$ of $R$ which is needed to grant the existence of $B.$ In
Proposition \ref{skew-grad} we prove a criterion under which the ring $B$ is Noetherian and flat as
a left and right $A$-module.\\

Then we generalise and axiomatise the localisation technique from \cite{cfksv} so that it applies
in our context. Assuming that $\delta(R)$ is contained in the Jacobson radical $ \Jac(R)$ of $R$ we
consider the canonical projection
\[\pi: A=R[[t;\sigma,\delta]]\twoheadrightarrow (R/\Jac(R))[[t;\overline{\sigma}]].\]
Theorem \ref{main-Ore}, the main result of section 2, states that $S=\C_A(\ker(\pi))$  is
an Ore set of $A$ consisting of regular elements. In particular, the localisation $A_S$ of $A$ with
respect to $S$ exists. The proof is again reduced to the Artinian case which is dealt with in
subsection 2.2 after discussing in  subsection 2.1 a general method how to attach Ore sets to a
ring homomorphism $R\to A$ of arbitrary rings $R$ and $A,$ which might be of its own interest.

As alluded to above the remarkable fact is that the two ring extensions $B$ and $A_S$ of $A$ are in fact closely related: The
filtration $I_{\bullet}$ induces a filtration on $A_S$ such that   the completion of $A_S$ with
respect to it is isomorphic to $B,$ see Proposition \ref{BasCompletion}, i.e.\ after completion it
is in some sense sufficient to just invert $t\in S$ instead of the much bigger set $S.$\\

In section 3 we apply our results to Iwasawa algebras. In this context the existence of the
localisation  $A_S$, of course, is known from \cite{cfksv}, see also \cite{ard-brown}. The main result is the
existence of the skew Laurent series ring $B$ in this setting and the fact that it is a
pseudocompact Noetherian ring, flat over $A$ and $A_S,$ see Theorem \ref{Bthm}.\\

Finally in section 4 we discuss different convergence conditions for our skew Laurent series which leads
to the definition of {\em overconvergent} skew Laurent series rings generalising again a concept
from $p$-adic Hodge theory as studied by Cherbonnier and Colmez. We expect that our constructions
and results will have important applications in non-commutative Iwasawa theory, but also in the theory of $p$-adic representations (see the forthcoming work of the first author with M-.F.\ Vigneras).

\section{ Infinite skew Laurent series rings }

Let $R$ be a  Noetherian pseudocompact ring together with the
following data:

\begin{enumerate}
\item  a topological ring automorphism $\sigma$ of $R,$
\item a continuous left $\sigma$-derivation $\delta:R\to R$ which is
$\sigma$-nilpotent in the sense of \cite[{\S 1}]{sch-ven}.
\end{enumerate}



In particular there exists the pseudocompact skew power series ring
\[A:=R[[t;\sigma,\delta]],\] see \cite[{\S}1]{sch-ven}, where the ring structure
is given by the following relation \be\label{relations} t
r=\sigma(r)t +\delta(r), \;\; (\mbox{ or by } rt=t\sigma'(r)
+\delta'(r))\ee and continuity. Here we set $\sigma':=\sigma^{-1}$
and
 $\delta':=-\delta\circ\sigma^{-1},$ i.e.\ $\delta'$ is a right $
\sigma'$-derivation.

The aim of this section is to introduce also a skew Laurent series
ring version which contains a formal inverse $t^{-1}$ of the element
$t.$  For non-trivial $\delta$ this involves in general series
$\sum_{i\in \z} r_i t^i$ with infinitely many nonzero coefficients
$r_i$ for negative $i$, as can be seen easily from the formulae
\eqref{rel-tinverse} below.

But if $\delta$ and $\delta'$ are nilpotent those formulae are
actually finite and we will show in the following subsection that
then the skew Laurent series ring with {\em finite negative part}
exists. Afterwards we shall construct for rather general $\delta$ a
skew Laurent series ring with {\em infinite negative part} but which
satisfies a certain convergence condition.

\subsection{Skew Laurent series rings in the nilpotent case}

In this subsection we assume in addition that both, $\delta$ and
$\delta'$ are nilpotent, say of degree $M,$ i.e.\
\[\delta^n\equiv 0,\;\; \delta'^n\equiv 0 \mbox{ for all } n\geq
M.\] Note that for a (left or right) Artinian ring $R$ the
nilpotence of $\delta$ and $\delta'$ follows already from
$\sigma$-nilpotence of $\delta.$ \\ For any integers $k,l \geq 0$,
$M_{k,l}(Y,Z)$ denotes the sum of all noncommutative monomials in
two variables $Y,Z$ with $k$ factors $Y$ and $l$ factors $Z.$
Consider the multiplicatively closed subset
\[T:=\{1,t, t^2,t^3,\cdots\}\] of $A$ consisting of all powers of
$t.$

\begin{prop}\label{skewLaurent}
The set $T$ satisfies the (right and left) Ore condition, i.e.\ the
localisation $A_T$ of $A$ at $T$ exists.
\end{prop}

\begin{proof}
For the right Ore condition we have to show that for any
$a=\sum_{i\geq 0} r_i t^i\in A$ and any $j<0$ there exist an element
$a'=\sum_{i\geq 0} r_i' t^i\in A$ and a natural number $N(j)$ such
that \be\label{ore} a t^{N(j)}=t^{-j}a'.\ee We first consider the
case $a=rt^0\in R$ and note that using the first relation in
\eqref{relations} and a simple telescope sum argument one checks the
identity
 \be\label{IH} t\cdot \sum_{-M\leq m\leq -1}
\sigma'\delta'^{-m-1}(r) t^{m+M} =r t^M
 \ee
in $A.$ On the other hand for a fixed $j<0$ we may find a natural
number $N(j)\gg -j$ such that $M_{j-m,l}(\delta',\sigma')\equiv 0$
for all $0 \leq l\leq 1-j$ and $m<-N(j).$ Then using the relation
\begin{equation*}
 M_{k+1,l}(Y,Z) - YM_{k,l}(Y,Z) = ZM_{k+1,l-1}(Y,Z)
\end{equation*}
one easily checks the identity
 \be\label{IS} t\cdot
\sum_{m\leq j} \sigma' M_{j-m,1-j}(\delta',\sigma')(r)
t^{m+N(j)}=\sum_{m\leq j+1}\sigma'
M_{(j+1)-m,1-(j+1)}(\delta',\sigma')(r) t^{m+N(j)}
 \ee
in $A;$ note that the sums indexed by `$m\leq j$' or `$m\leq j+1$'
are actually finite. Using \eqref{IH} as starting point and
\eqref{IS} as induction step it follows that
 \be\label{IR}
 t^{-j}\cdot
\sum_{m\leq j}\sigma' M_{j-m,1-j}(\delta',\sigma')(r) t^{m+N(j)}
=rt^{N(j)}.
 \ee
For an arbitrary $a\in A$ we set
\[a':=\sum_{i\geq 0}\Big(\sum_{m\leq j}
\sigma'M_{j-m,1-j}(\delta',\sigma')(r_i)t^{m+N(j)}\Big) t^i\in A.\]
Since the multiplication  in the pseudocompact ring $A$ is
continuous the multiplication from the left by $t^{-j}$ commutes
with the (first)   summation and hence we obtain $\eqref{ore}.$ The
left Ore condition follows from an analogous argument and thus the
localisation exists by \cite[thm.\ 2.1.12]{mc-rob} and the fact that
all elements in $T$ are regular.
\end{proof}

Clearly all elements in the {\em skew Laurent series ring}
\[B := R((t;\sigma,\delta)):=A_T\] can be written as series
$\sum_{i\gg-\infty} r_it^i$ and  $\sum_{i\gg-\infty} t^ir_i$ with
{\em finite} negative part $\sum_{-\infty<i<0}r_it^i$ and
$\sum_{-\infty<i<0}t^ir_i,$ respectively.

\begin{rem}
The same argument shows that $T$ is also an Ore set in the skew
polynomial ring $R[t;\sigma,\delta]$ and thus also the {\em skew
Laurent polynomial} ring
\[R(t;\sigma,\delta):=R[t;\sigma,\delta]_T\] exists under the
hypothesis of this subsection.
\end{rem}


It follows from the identity \eqref{IH} that in the ring
$R((t;\sigma,\delta))$ the   relations \be\label{rel-tinverse}
t^{-1}r=\sum_{i\leq -1} \sigma'\delta'^{-i-1}(r) t^i,\ee and \be r
t^{-1}=\sum_{i\leq -1} t^i\sigma\delta^{-i-1}(r) \ee hold.    More
generally, for $j<0,$ the identity \eqref{IR} says that \be t^j
r=\sum_{m\leq j} \sigma' M_{j -m,1-j}(\delta',\sigma')(r) t^m\ee and
\be rt^j =\sum_{m\leq j} t^m \sigma M_{j
-m,1-j}(\delta,\sigma)(r).\ee For $j\geq 0$ we recall the formulae
\be t^j r=\sum_{0\leq m\leq j} M_{j-m,m}(\delta,\sigma)(r) t^m\ee
and \be rt^j =\sum_{0\leq m\leq j}t^m M_{j-m,m}(\delta',\sigma')(r)
.\ee Finally, the multiplication in $R((t;\sigma,\delta))$ and in
$R(t;\sigma,\delta)$ is explicitly given by the following   formula
\be \label{mult}(\sum_{j\in\z} a_j t^j)(\sum_{l\in\z} b_l
t^l)=\sum_{m\in\z}c_m t^m \;\;\;\;\mbox{ with }\ee
\begin{eqnarray}
 c_m&:=&c_m^++c_m^-,\\ \label{cm+} c_m^+&:=& \sum_{j\geq n\geq 0} a_j
M_{j-n,n}(\delta,\sigma)(b_{m-n}) \;\;\mbox{  and }\\ \label{cm-}
c_m^-&:=&\sum_{ n\leq j< 0}
a_j\sigma'M_{j-n,1-j}(\delta',\sigma')(b_{m-n}).\end{eqnarray}  An
analogous formula holds for {\em right} Laurent series: \be
\label{mult-right}(\sum_{j\in\z}t^j a_j )(\sum_{l\in\z} t^lb_l
)=\sum_{m\in\z}t^m d_m \;\;\;\; \ee with
\begin{eqnarray}
 d_m&:=&d_m^++d_m^-,\\ \label{cm+right} d_m^+&:=& \sum_{j\geq n\geq
 0}
M_{j-n,n}(\delta',\sigma')(a_{m-n}) b_j\;\;\mbox{  and }\\
\label{cm-right} d_m^-&:=&\sum_{ n\leq j< 0}  \sigma
M_{j-n,1-j}(\delta,\sigma)(a_{m-n})b_j.\end{eqnarray}

For the notion of the Krull-dimension $\kappa$ we refer the reader
to \cite[chap.\ 6]{mc-rob}. If we consider a ring $S$ as left or
 right module over itself we write $_SS$ or $S_S$, respectively. For a
 left (right) module $M$ we denote by
 $\mathcal{L}(M)$ the lattice of all left (right)
 submodules of $M.$

\begin{lem}\label{Art} Suppose that $\delta = 0$. We then have
$\kappa(B_B)\leq \kappa(R_R )$ and $\kappa(_BB)\leq \kappa(_RR)$. In
particular, if $R$ is a (left or right) Artinian ring, so is $B$.
\end{lem}

\begin{proof}
The `left'-version following by symmetry we only consider the
`right'-version. For the purposes of this proof we call $i(b) :=
\min\{i: a_i \neq 0\}$ the order of any element $b =
\sum_{i\gg-\infty} a_i t^i$ in $B$ and $l(b) := a_{i(b)}$ its
leading coefficient (with $l(0) := 0$). If $J$  is a right ideal
 of $B$ we write $l(J)$ for the set of leading coefficients of
 elements in $J.$  Using that $\delta=0$ it is easy to check that $l(J)$ is a right ideal
 of $R,$ thus we obtain a map of partially ordered sets
 $l: \mathcal{L}(B_B)\to \mathcal{L}(R_R)$.
We just have to show that this map $l$ preserves `proper
containment'. Thus let $J_1\subsetneqq J_2$ be  right ideals of $B$.
 Then there exists an element $b\in
J_2\cap A$ which is not contained in $J_1.$ We shall derive a
contradiction   assuming that $l(J_1)=l(J_2):$  then we find $j_1\in
J_1 $ with $l(j_1)=l(b)$ and, after possibly multiplying   from the
right by a suitable power of $t,$ with $i(j_1)=i(b);$ in particular
$j_1$ belongs to $J_1\cap A.$ Similarly we find $j_2\in J_1\cap A$
such that $l(j_2)=l(b-j_1)$ and $i(j_2)=i(b-j_1)$ and inductively a
sequence of elements $j_n\in J_1\cap A$ with strictly increasing
order and such that the series $\sum_{n\geq 1} j_n$ converges to
$b,$ which thus is an element of the closed ideal $J_1\cap A$ ($A$
is pseudocompact and Noetherian), a contradiction.
\end{proof}

Let $I$ be a two-sided ideal of $R$ which is $\sigma$-, $\sigma'$-
and $\delta$-stable. We define the left $R$-submodule $I_B$ of $B$
to consist of all $b=\sum b_i t^i\in B$ with $b_i \in I$ for all
$i\in \z.$

\begin{lem}\label{twosided} Let $I,J$ be $\sigma$-,   $\sigma'$- and
$\delta$-stable two-sided ideals of $R.$ Then
 \begin{enumerate}
 \item $I_B=IB= BI$ is a two-sided ideal of $B.$
   \item $J_B\cdot I_B= (JI)_B.$
 \end{enumerate}
\end{lem}

\begin{proof}(i) Since $I$ is finitely generated as a right ideal in $R$ we have
$I_B=IB$. The formula \eqref{mult} implies that $BI  \subseteq I_B =
I B$. By symmetry we must even have $BI = IB$.
 (ii) follows  immediately from (i).
\end{proof}

In \cite[{\S}1]{sch-ven} we constructed a descending exhaustive ring
filtration $I_k,$ $k\geq 0,$  by two-sided ideals in $R$ which are
$\sigma$- and $\sigma'$-stable and satisfy $\delta(I_k)\subseteq
I_{k+1},$ which we refer to as the {\em standard} filtration. By
Lemma \ref{twosided} the filtration $I_\bullet$ of $R$ induces a
ring filtration $J_\bullet$ of $B$ by setting $J_k :=(I_k)_B$ for
all $k\geq 0.$

Henceforth we assume that $R$ is (left or right) Artinian. Then the filtration $I_\bullet$ and thus
also $J_\bullet$ stabilises. Furthermore, as $R$ is Noetherian, the $J_k$ are finitely generated
(left and right) $B$-modules by Lemma \ref{twosided}. Hence the subquotients $J_k/J_{k+1}$ are
finitely generated modules over $B/J_1 = R/I_1 \otimes_R B \cong (R/I_1)((t;\bar{\sigma}))$
 and thus have
finite length by Lemma \ref{Art}. We have shown the following

\begin{prop}\label{Art-general}
Suppose that the standard filtration is separated, i.e.\ $I_k=0$ for
$k\gg0.$  Then, if $R$ is (left or right) Artinian, so is $B.$
\end{prop}

\begin{rem}\label{standardversusany} Let $R$ be (left or right) Artinian.
Note that the standard filtration is separated if and only if there
exist {\em any}    separated descending  ring filtration $I_k,$
$k\geq 0,$  by two-sided ideals in $R$ which are $\sigma$- and
$\sigma'$-stable and satisfy $\delta(R)\subseteq I_{1}.$
\end{rem}

\begin{proof}
 Assume that such a filtration $I_\bullet$ is given and let $k$ be any natural number. For $l\geq 1$ the  $l$th ideal of
 the standard filtration is generated by all subgroups \be\label{standard}M_{m_1}(R)\cdot \ldots \cdot
 M_{m_r}(R)\ee where  $m_1+\ldots + m_r=l$ is any partition of $l$ with $m_j>0,$ $1\leq j\leq r,$ and
 $M_{m_j}$ is a non-commutative monomial in $\delta,$ $\sigma$ and $\sigma'$ with at least $m_j$ factors
 $\delta$. Since $R$ is Artinian and due to the $\sigma$-nilpotence
 of $\delta$ we find a positive number $m$ such that $M(R)=0$ for
 all such monomials $M$ with at least $m$ factors $\delta.$ Choosing $l>
 km$ we see that the $l$th step of the standard filtration is
 contained in $I_k$   because the only non-zero contributions of the
 form
 \eqref{standard} have at least $k$ factors, which all belong  to
 $I_1$ by assumption.
\end{proof}

\subsection{Skew Laurent series rings with infinite negative part}

The following definition leads to a reasonable, in both directions
infinite, skew Laurent series ring: Let
\[B:=R\hspace{-3 pt}\ll\hspace{-2 pt} t;\sigma,\delta]]\] consist of all formal
infinite sums $\sum_{i\in \z} r_i t^i$ such that $r_j$ tends to zero
in the pseudocompact topology of $R$ for $j\leq 0$ running to
$-\infty.$

$B$ is naturally endowed with the exhaustive and separated
descending filtration $(F^kB)_{k\geq 0}$ of left $R$-modules defined
as
\[F^kB:= \big(\prod_{i\in\z} \Jac(R)^kt^i\big)\cap B,\]
where $\Jac(R)$ denotes the Jacobson radical of $R.$  The topology
induced by this filtration will be called the {\em strong} topology:
There is another interesting topology on $B$ but which will not be
used in this paper. This {\em weak} topology is given by the system
of open zero neighbourhoods $\{F^kB + At^m\}_{k,m \geq 0}$.

\begin{rem}
 $B$ is a complete $R$-module with respect to the strong topology.
\end{rem}

Note that $\sigma(\Jac(R))=\Jac(R),$ in particular the
$\Jac(R)$-adic filtration is $\sigma$- and $\sigma'$-stable.  In
general we are not assuming that $\Jac(R)$ is also stable under
$\delta,$ but the continuity of $\delta$ implies that there is a
natural number $s\geq 1$ such that \be
\delta(\Jac(R)^s),\;\;\delta'(\Jac(R)^s)\subseteq \Jac(R). \ee By
induction one shows immediately that \be
\delta(\Jac(R)^{sk}),\;\;\delta'(\Jac(R)^{sk})\subseteq \Jac(R)^k
\ee for all $k\geq 0.$ If $M(Y,Z)$ denotes a noncommutative monomial
with $m$ factors $Y$ and arbitrarily, but finitely many factors $Z$,
we obtain \be\label{stetig}
M(\delta,\sigma)(\Jac(R)^{s^mk}),\;\;M(\delta',\sigma')(\Jac(R)^{s^mk})\subseteq
\Jac(R)^k\ee for all $k\geq 0.$

%
%
%
%

Let $I_n,$ $n\geq 0,$ be a $\sigma$-,   $\sigma'$- and
$\delta$-stable separated exhaustive descending filtration
  of $R,$ in particular satisfying $I_k\cdot
I_l\subseteq I_{k+l},$ consisting of (closed two-sided) ideals.
We define the exhaustive and separated filtration
\[J _k:= \big(\prod_{i\in\z} I_kt^i\big)\cap B \] of $B$ consisting
of strongly closed left $R$-submodules.\\

\noindent {\bf Assumption (I):} There exists a filtration $(I_k)_k$
as above
where the ideals $I_k$ are all {\em open} in $R.$\\

A slightly stronger version is the following\\

\noindent {\bf Assumption ($\mathrm{\mathbf{SI_0}}$):} There exists
a filtration $(I_k)_k$ as above
where the ideals $I_k$ are all {\em open} in $R$ and such that $\delta(R)\subseteq  I_{ 1}.$\\

Later we shall also consider the strong version in which the analogous condition holds in all degrees.\\

\noindent {\bf Assumption (SI):} There exists a filtration $(I_k)_k$
as above
where the ideals $I_k$ are all {\em open} in $R$ and such that $\delta(I_k)\subseteq  I_{k+1}$ for all $k\geq 0.$\\

\begin{rem}\label{Jac}
If $\delta(\Jac(R))\subseteq \Jac(R)$ holds, then $\mathrm{(I)}$ is
satisfied (with $I_\bullet$  the $\Jac(R)$-adic filtration). If, in
addition, $\delta(R)\subseteq \Jac(R)$ or even both
$\delta(R)\subseteq \Jac(R)$ and $\delta(\Jac(R))\subseteq
\Jac(R)^2$ hold, then $(\mathrm{SI}_0)$ and $(\mathrm{SI})$ are
satisfied, respectively.
\end{rem}

If we assume $\mathrm{(I)}$, then the filtrations $J_k$ and $F^kB$
are compatible, in particular also the filtration $(J_k)_k$ induces
the strong topology on $B$ and  we obtain an isomorphism
\be\label{prolim} B\cong \projlim{k} B/J_k\ee of topological
$R$-modules. On the other hand   we then have an isomorphism of left
$R$-modules \be\label{BmodI} B/J_k\cong (R/I_k)
((t;\overline{\sigma},\overline{\delta})),\ee where
$\overline{\sigma},$ $\overline{\sigma'},$ $\overline{\delta}$ and
$\overline{\delta'}$ denote the induced maps on $R/I_k.$ Note that
since $I_k$ is open in $R$  the $\sigma$-nilpotence of $\delta$
implies that both $\overline{\delta}$ and $\overline{\delta'}$ are
nilpotent and thus the latter ring exists by proposition
\ref{skewLaurent}.

Below we shall show that the formula \eqref{mult} defines an
obviously distributive multiplication law on $B,$ which, for every
$k,$ induces by construction the ring structure of
$(R/I_k\big)((t;\overline{\sigma},\overline{\delta}))$ and thus
coincides with the ring multiplication of the projective limit ring.
In particular, the multiplication law on $B$ is also associative.

\begin{prop}\label{skewLaurentinfinite} If $\mathrm{(I)}$ is satisfied,
 then the formula \eqref{mult} defines a topological ring structure
on $B$ with respect to  the strong topology. Moreover, if {\rm
($\mathrm{SI}_0$)} holds, then $B$ is a pseudocompact ring.
\end{prop}

\begin{proof} We will first show that \eqref{mult} (actually without
even assuming $\mathrm{(I)}$) gives a well-defined map $B\times B\to
B.$ To this end we check for the positive and negative parts $c_m^+$
and $c_m^-$ separately, that the defining sums in \eqref{cm+} and
\eqref{cm-} converge, independently of  the order of summation, and
that for $m$ tending to $-\infty$ the $c_m^\pm$ converge to zero:
Let $k$ be any given positive number and $m$ be any fixed integer.
Then by the $\sigma$-nilpotence of $\delta$ there is a
  constant $N_1\gg0$ such that $ M_{j-n,n}(\delta,\sigma)(R)$
  (and for later purposes $M_{j-n,1-j}(\delta',\sigma')(R)$) and thus
  also
$a_jM_{j-n,n}(\delta,\sigma)(b_{m-n})$ lies in $\Jac(R)^k$ for all
$j-n\geq N_1.$ On the other hand, by the definition of $B,$ there
exists a   constant $N_2\gg0$ such that $b_{m-n}\in
\Jac(R)^{s^{N_1}k}$ for all $m-n\leq -N_2.$  Thus it follows from
\eqref{stetig} that the summand
$a_jM_{j-n,n}(\delta,\sigma)(b_{m-n})\in \Jac(R)^k$ whenever $n\geq
N_2+m$ or $j-n\geq N_1.$ Hence all but possibly  the finitely many
summands of \eqref{cm+} for $n\leq j < N_1+N_2 +m$ and $0\leq n <
N_2 +m$ lie in $\Jac(R)^k.$ This implies the convergence of the
positive part. Now we will show that $c_m^+$ belongs to $\Jac(R)^k$
if $m$ is small enough: We have already seen that all summands
outside this finite set of exceptions lie in $\Jac(R)^k.$ Now we
assume $m\leq -N_2$. Then $n \geq N_2 + m$ for any $n \geq 0$ and so
the exceptional set is empty.

For the negative part we assume $m$ again to be fixed but arbitrary.
By the definition of $B$ there exists a   constant $N_0\gg0$ such
that $a_j$ and thus the summand
$a_j\sigma'M_{j-n,1-j}(\delta',\sigma')(b_{m-n})$ lies in
$\Jac(R)^k$ for all $j\leq -N_0.$ On the other hand    we have
$M_{j-n,1-j}(\delta',\sigma')(b_{m-n})\in \Jac(R)^k$ for all
$j-n\geq N_1,$ see above. Thus apart from possibly the indices
$-N_0-N_1 < n < 0$ all summands belong to $\Jac(R)^k$ which implies
convergence of the negative part. Moreover,  assuming $m\leq
-N_0-N_1-N_2$  we have -   as for the positive part - $b_{m-n}\in
\Jac(R)^{s^{N_1}k}$ whence $M_{j-n,1-j}(\delta', \sigma'
)(b_{m-n})\in\Jac(R)^k$ for all $n > -N_0-N_1.$ It follows that
$c_m^-\in\Jac(R)^k.$

Finally, we show that the multiplication is  continuous with respect
to the strong topology. As the addition is continuous and  the
multiplication is distributive it suffices to check this in a
neighbourhood of $0.$ But from the formula \eqref{mult} one sees
that $F^kB\cdot B\subseteq F^kB$ for any given $k\geq 0.$ By
\eqref{BmodI}, \eqref{prolim}, Remark \eqref{standardversusany} and
Proposition \ref{Art-general} the ring $B$ is pseudocompact provided
{$\mathrm{ (SI_0)}$} holds.
 \end{proof}

\begin{rem}\label{Bprolimes} Assume $\mathrm{(I)}$ (respectively ($\mathrm{SI}_0$)). Then,
a posteriori the isomorphism \be \label{topiso} B\cong
\projlim{k}\big(R/I_k\big)((t;\overline{\sigma},\overline{\delta}))\ee
induced by \eqref{prolim} and \eqref{BmodI} is an isomorphism of
{\em topological (pseudocompact) rings} if the rings
$(R/I_k\big)((t;\overline{\sigma},\overline{\delta}))$ are endowed
with the discrete   topology. Moreover, the $J_k$ are two-sided
ideals of $B$ because they are kernels of the natural ring
homomorphisms
$B\to(R/I_k\big)((t;\overline{\sigma},\overline{\delta})).$
\end{rem}

Henceforth we assume that ($\mathrm{I}$) holds. Let $I$ be a
two-sided ideal of $R$ which is $\sigma$-, $\sigma'$- and
$\delta$-stable.   For $C = A$ or $B$ we define as
before the left $R$-submodule $I_C$ of $C$ to consist of all $c=\sum
c_i t^i\in C$ with $c_i \in I$ for all   $i \geq 0$,
resp.\ $i\in \z$.

\begin{lem} \label{twosided-general} Let $I,J$ be $\sigma$-,   $\sigma'$- and
$\delta$-stable two-sided ideals of $R$   and let $C = A$
or $B$.  Then
 \begin{enumerate}
 \item $I_C=IC= CI$ is a two-sided ideal of $C.$
   \item $J_C\cdot I_C= (JI)_C.$
 \end{enumerate}
 In particular, $B$ is a filtered ring with respect to $J_\bullet.$
\end{lem}
\begin{proof}
  For $C = A$ the proof is identical with the one of
Lemma \ref{twosided}. In case $C = B$ we have to modify the former
argument.  Fix generators $u_1,\ldots,u_m$ of $I$ as a right ideal
in $R$. Since $R$ is pseudocompact we find a strictly increasing
sequence of natural numbers $j(1) < j(2) < \ldots$ such that
 \[ \sum_{i=1}^m
u_i\Jac(R)^k \supseteq I \cap \Jac(R)^{j(k)} \qquad \textrm{for any
$k \geq 1$.}
 \]
  This implies that $I_B = IB:$ for $b=\sum b_i t^i\in I_B$
  we can  write $b_i=\sum_{j=1}^m u_j c_i^{(j)},$  $i\in\z,$ such that
  $c^{(j)}:= \sum_{i\in\z} c_i^{(j)}t^i$ belongs to $B$, whence
  $b=\sum_{j=1}^m u_j c^{(j)}\in IB;$
  the other inclusion is obvious. The rest of the proof is exactly the
same as the one of Lemma \ref{twosided}.
\end{proof}



\begin{rem}
Using the construction in \cite[chap.\ IV\S1]{li-o} it is not difficult to show that the filtered
ring $(B,J_\bullet)$ is the algebraic microlocalisation of the filtered ring $(A,J_\bullet \cap A)$
in the multiplicative subset $\{1,t,t^2,\ldots\}$.
\end{rem}

\begin{lem}\label{gr}
  We have
  \[\gr_{J_\bullet} B\cong
   \gr_{I_\bullet} R \otimes_{R/I_1} (R/I_1)((t;\overline{\sigma},\overline{\delta}))
   \cong (R/I_1)((t;\overline{\sigma},\overline{\delta}))
 \otimes_{R/I_1} \gr_{I_\bullet} R \]
 where the ring multiplication on the right hand side is given by
 formulae \eqref{mult} and \eqref{mult-right}, respectively, if we
 view the elements in the right side as Laurent series in the variable
 $t$ over the ring $\gr_{I_\bullet} R$.
\end{lem}
\begin{proof} Under the isomorphism  \eqref{BmodI} the ideal  $J_{k-1}/J_{k}$
of $B/J_k$ corresponds to the ideal
$(I_{k-1}/I_k)_{(R/I_k)((t;\overline{\sigma},\overline{\delta}))}$
which is isomorphic to   $(I_{k-1}/I_k) \otimes_{R/I_1} (R/I_1)((t;
\overline{\sigma},\overline{\delta}))$ as $I_{k-1}/I_k$ is finitely
generated over $R/I_1.$ The result follows.
\end{proof}

\begin{lem}
 $B$ is faithfully flat as a left or right $R$-module.
\end{lem}
\begin{proof}
By symmetry it suffices to consider the left module case. We have to
show that, for any proper right ideal $Q \subset R$, the natural map
$Q \otimes_R B \longrightarrow B$ is injective but not surjective.
The non-surjectivity is clear. To establish the injectivity we fix
generators $u_1,\ldots,u_m$ of the right ideal $Q$. Then any element
$x \in Q \otimes_R B$ can be written as
 \[ x = \sum_{j=1}^m u_j \otimes b^{(j)} \qquad\textrm{with}\ \ b^{(j)}
 = \sum_{i \in \z} b_i^{(j)}t^i \in B.
 \]
We suppose now that the image $\sum_j u_j b^{(j)}$ of $x$ under the above map is zero. Then the
tuple $(b_i^{(1)},\ldots,b_i^{(m)})$, for any $i \in \z$, lies in the right submodule $N :=
\{(a^{(1)},\ldots,a^{(m)}) \in R^m : \sum_{j=1}^m u_ja^{(j)} = 0\}$ of $R^m$. Since $R$ is
Noetherian $N$ has finitely many generators $\alpha_1 = (a_1^{(1)},\ldots,a_1^{(m)}),
\ldots,\alpha_s = (a_s^{(1)},\ldots,a_s^{(m)})$. Write
 \[ (b_i^{(1)},\ldots,b_i^{(m)}) =
 \alpha_1 c_i^{(1)} + \ldots
 + \alpha_s c_i^{(s)} \qquad\textrm{for any}\ \ i \in \z \]
with $c_i^{(1)},\ldots,c_i^{(s)} \in R$. In fact, since by the
pseudocompactness of $R$ we have a strictly increasing sequence $0 <
j(1) < j(2) < \ldots$ such that
 \[ \sum_{k=1}^s \alpha_k \Jac(R)^\ell \supseteq N \cap
 (\Jac(R)^{j(\ell)}R^m) \qquad\textrm{for any}\ \ \ell \geq 1 \]
we may choose the $c_i^{(k)}$ in such a way that $c^{(k)} := \sum_{i
\in \z} c_i^{(k)}t^i$ lies in $B$ for each $1 \leq k \leq s$. Then
$b^{(j)} = \sum_{k=1}^s a_k^{(j)}c^{(k)}$ in $B$ and hence
 \[ x = \sum_{j=1}^m u_j \otimes b^{(j)} = \sum_{j=1}^m u_j \otimes
 \sum_{k=1}^s a_k^{(j)}c^{(k)} = \sum_{k=1}^s (\sum_{j=1}^m
 u_ja_k^{(j)}) \otimes c^{(k)} = 0. \]
\end{proof}

\begin{prop}\label{skew-grad}
Assume that $(\mathrm{SI})$ holds, that $\sigma$ induces the identity map on  $\gr_{I_\bullet} R$,
and that $\gr_{I_\bullet} R$ is an almost normalising extension of its subring $R/I_1$ in the sense
of \cite[1.6.10]{mc-rob}. Then $B=R\lk t;\sigma,\delta]]$ is Noetherian and is flat as a left and a
right $A$-module.
\end{prop}

\begin{proof} By  $(\mathrm{SI})$ we have
$\overline{\delta}=0$. Using Lemma \ref{gr} we see that
$\gr_{J_\bullet} B\cong
   \gr_{I_\bullet} R \otimes_{R/I_1} (R/I_1)((t))$ is a subring of
the Laurent series ring $(\gr_{I_\bullet} R)((t))$ and is an almost normalising extension of
$(R/I_1)((t))$. The latter is well-known to be Noetherian. Hence the former is Noetherian as well
by \cite[thm. 1.6.14]{mc-rob}. The first assertion now follows from \cite[prop.\ II.1.2.3]{li-o}.
Similarly we have $\gr_{J_\bullet \cap A} A\cong
   \gr_{I_\bullet} R \otimes_{R/I_1} (R/I_1)[[t]]$ which is a
Noetherian ring as well. Moreover it follows that $\gr_{J_\bullet} B$ is the localisation of
$\gr_{J_\bullet \cap A} A$ in $t$ which therefore is a flat ring extension. The flatness of $B$
over $A$ now follows from \cite[prop.\ II.1.2.1 and prop.\ II.1.2.3]{li-o}.
%
%
\end{proof}

\section{Another canonical Ore set}

Keeping the notations and assumptions of the previous section we
assume in addition throughout this section that \[\delta(R)\subseteq
\Jac(R).\]   Then, by Remark \ref{Jac}, the assumption
$(\mathrm{SI}_0)$ is satisfied for the $\Jac(R)$-adic filtration.

Note that $\Jac(R)_A$ is the kernel of the canonical projection
\[A=R[[t;\sigma,\delta]]\twoheadrightarrow (R/\Jac(R))[[t;\overline{\sigma}]].\]

 The aim of this section is to show that
 \begin{eqnarray}
 \label{Ore-reg} \C_A(\Jac(R)_A) \mbox{ is an Ore set of $A$ consisting of regular elements,}
 \end{eqnarray}
   i.e. that the localisation
  $A_{\C_A(\Jac(R)_A)}$ of $A$ with respect to $\C_A(\Jac(R)_A)$ exists. In  the first subsection we shall
  establish a general method how to   attach to a ring homomorphism $R\to A$ of arbitrary
  rings $R$ and $A$ a  left (respectively right) Ore set $S_l$ ($S_r$). In the second subsection we
  return to our above setting under the additional hypothesis that $R$ is Artinian (and later
  semisimple and even simple): we show statement \eqref{Ore-reg} and that $\C_A(\Jac(R)_A)$ equals $S_l$ and $S_r.$ In
  the third and last subsection  we use the results of the previous ones to lift \eqref{Ore-reg} to the
  general case.

  \subsection{Ore sets attached to homomorphisms of rings}

Let $\alpha: R\to A$ be a homomorphism of (unital) rings. A left or
right ideal $I$ of $A$ is called {\em $R$-cofinite} if $A/I$ is a
  Noetherian $R$-module (via $\alpha$.) We define the
set $S_l$ of left $R$-cofinite elements of $A$ as
\[S_l:=S_l(\alpha):=\{s\in A |\; As \mbox{ is $R$-cofinite}\}.\] The
set $S_r:=S_r(\alpha)$ of right   $R$-cofinite elements
of $A$ is  defined analogously.

\begin{lem}\label{Smult}
\begin{enumerate}
  \item  $S_l$ and $S_r$ are multiplicatively closed.
  \item If $as \in S_l$ (resp.\ $sa \in S_r$) then $s \in S_l$ (resp.\
  $s \in S_r$).
\end{enumerate}
\end{lem}

\begin{proof}
We only discuss the ``left'' versions. The statement (i) follows
immediately from the following exact sequence
\[
A/As \mathop{\longrightarrow}\limits^{\cdot t} A/Ast
\mathop{\longrightarrow}\limits^{\rm pr} A/At \longrightarrow 0\ .
\]
The statement (ii) is a consequence of the surjection
$A/Aas\twoheadrightarrow A/As$.
\end{proof}

Consider the following property \\

\noindent $\bf{ {PC}_l(\alpha)}$\textbf{:}
 Any $R$-cofinite left ideal $I$ of $A$ is {\em
principally cofinite}, i.e.\ contains a
 principal left $R$-cofinite ideal or, in other words, an
element in $S_l.$\\

The property $\bf  {PC}_r(\alpha)$ is defined similarly. In the
following we only give ``left'' versions of assertions and proofs.
But it is understood that in each case the corresponding ``right''
version holds as well.

\begin{lem} Assume $PC_l(\alpha).$ Then, for any $s$ in  $S_l$ the left
$A$-module $A/As$ is $S_l$-torsion.
\end{lem}

\begin{proof}
 Let $c \in A$ be any element and consider the left ideal
 $L := \{b \in A : bc \in As\}$ in
$A$. Since $A/L \mathop{\longrightarrow}\limits^{\cdot c} A/As$ is
injective the left ideal $L$ is $R$-cofinite. Our assumption
$PC_l(\alpha)$ therefore ensures the existence of an element $t \in
S_l \cap L$. Then $t(c + As) = 0$.
\end{proof}

\begin{prop}Assume $PC_l(\alpha)$ and let $M$ be a Noetherian left $A$-module;
then $M$ is Noetherian as an $R$-module if and only if it is
$S_l$-torsion.
\end{prop}

\begin{proof}
Since $S_l$ is multiplicative the extension of two $S_l$-torsion
$A$-modules is $S$-torsion as well. Secondly, since $M$ is
Noetherian over $A$ there is a finite exhaustive and separated
filtration on $M$ whose subquotients are cyclic (and Noetherian)
$A$-modules. These observations reduce us to the case where $M$ is
cyclic over $A$, i.e.\ of the form $M \cong A/L$ for some left ideal
$L$ of $A$. If $M$ is $S_l$-torsion we find an $s \in S_l$ such that $s
= s\cdot 1 \in L$. Then $M$ is a quotient of the Noetherian
$R$-module $A/As$ and therefore is Noetherian over $R$. Suppose,
vice versa, that $M$ is Noetherian over $R$ which means that $L$ is
$R$-cofinite. By our assumption  $PC_l(\alpha)$ we have $As
\subseteq L$ for some $s \in S_l$. According to the above lemma $A/As$
and a fortiori $A/L \cong M$ are $S_l$-torsion.
\end{proof}

\begin{prop}\label{pc-ore}
If $\mathrm{PC}_l(\alpha)$ is satisfied, then $S_l$ is a left Ore
set of $A$.
\end{prop}

\begin{proof}
  Let $s \in S$ and $b \in A$. By the lemma the $A$-module $A/As$ is $S$-torsion. Hence $tb
\in As$ for some $t \in S$. We therefore find a $b' \in A$ such that $tb = b's$.
\end{proof}

\begin{lem}\label{PC-product}
   Let $\alpha_i : R_i\to A_i$, for $1\leq i\leq m,$ be homomorphisms of
 arbitrary (unital) rings and let $\alpha : R:=R_1\times \ldots \times
 R_m \to A:=A_1\times \ldots\times A_m$ be their product. Then $PC_l(\alpha)$
 holds if and only if $PC_l(\alpha_i)$ holds for all $1\leq i\leq
 m.$
\end{lem}

\begin{proof}
  Without loss of generality we assume $m=2.$ Then  every
left ideal $L$ of $A$ is of the form $L_1\times L_2$ with left
ideals $L_i=e_iL$ of $A_i,$ where $e_i$ denotes the central
idempotent  corresponding to $A_i.$ Thus $A/L=A_1/L_1\times A_2/L_2$
is Noetherian over $R$ if and only if  each $A_i/L_i=e_i(A/L)$ is
Noetherian over $R_i=e_iR.$ Now assume that $PC_l(\alpha_i)$ holds
for $i=1,2$ and let $L$ be an $R$-cofinite left ideal of $A.$ Then
$L_i$ contains an element $f_i$ such that $A_if_i$ is $R_i$-cofinite
for $i=1,2.$ Putting $f:=(f_1,f_2)\in L_1\times L_2=L$ we obtain an
$R$-cofinite ideal $Af\subseteq L,$ whence $PC_l(\alpha)$ follows.
For the converse let $L_1$ be an $R_1$-cofinite left ideal of $A_1.$
Then $L:=L_1\times A_2$ is an $R$-cofinite left ideal of $A$ which
thus by assumption contains an $R$-cofinite element $f.$ Then
$f_1=e_1f\in e_1L=L_1$ is $R_1$-cofinite. Thus  $PC_l(\alpha_1),$
and similarly  $PC_l(\alpha_2),$  follows.
\end{proof}

\begin{lem}\label{PC-functorial} Let $R \xrightarrow{\beta}
A_0 \xrightarrow{\gamma} A_1$ be ring homomorphisms and put $\alpha
:= \gamma\circ\beta$.
\begin{enumerate}
 \item Suppose that we have an surjection $A_0^m\twoheadrightarrow A_1$ of
 $A_0$-bimodules
 and $PC_l(\beta)$ holds; then $PC_l(\alpha)$ holds and
 \be\label{sat} S_l(\alpha)=\{s\in A_1\;|\; as\in \gamma(S_l(\beta)) \mbox{ for
 some } a\in A_1\}.\ee In particular, if $A_0\subseteq A_1$ then
 $S_l(\beta)\subseteq S_l(\alpha).$
 \item If $A_0$ is Noetherian as a left $R$-module, then $PC_l(\alpha)$ holds
 if and only if $PC_l(\gamma)$ holds.
\end{enumerate}
\end{lem}

\begin{proof}
Let $L$ be an $R$-cofinite left ideal of $A_1$. Then
$A_0/\gamma^{-1}(L) \subseteq A_1/L$ are Noetherian $R$-modules.
Thus property $PC_l(\beta)$  grants the existence of an $f\in
\gamma^{-1}(L)$ such that $A_0/A_0f$ is Noetherian over $R.$ Hence
$A_1/A_1\gamma(f)=A_1\otimes_{A_0} A_0/A_0f$ is the image of
$(A_0/A_0f)^m$ for some $m$ and thus is Noetherian   over $R.$ This
proves the first part of (i).
  For \eqref{sat} we  first assume that $
\gamma(f) = as $ belongs to $  \gamma(S_l(\beta)).$ Then by the same
argument as above $A_1/A_1as=A_1\otimes_{A_0} A_0/A_0f$ is
  Noetherian over $R.$ Hence $as \in S_l(\alpha)$ and
therefore $s \in S_l(\alpha)$ by Lemma \ref{Smult}(ii).  For the
converse suppose that $s\in S_l(\alpha).$ Since
$A_0/\gamma^{-1}(A_1s)$ is an $R$-submodule of the   Noetherian left
$R$-module $A_1/A_1s,$ the left ideal $\gamma^{-1}(A_1s)$ of $B$ is
$R$-cofinite. Thus property $PC_l(\beta)$ implies that
$ \gamma^{-1}(A_1s)\cap S_l(\beta)\neq \emptyset.$\\
Under the hypothesis of (ii) any finitely generated (left)
$A_0$-module $M$ is Noetherian over $R$, hence the statement is
clear.
\end{proof}

\subsection{ The Artinian case}

In this subsection we assume that, in addition to our standard
hypothesis, $R$ is (left and right) Artinian. Then $B=R((t;\sigma,
\delta))$ is Artinian, too, by Proposition \ref{Art-general} and
Remark \ref{standardversusany} applied to the $\Jac(R)$-adic
filtration of $R.$ Hence, by \cite[prop.\ 3.1.1]{mc-rob} $B$ is a
(left and right) quotient ring, i.e.\  every regular element of $B$
is a unit. Since $t$ is regular in $A,$ we have a canonical
inclusion $A\subseteq A_T=B.$ The regular elements of $A$ are also
regular in $A_T,$ they thus are all units   in $B$. It
easily follows that  $\C_A(0)$ is an Ore set and   that
\[ B=A_{\C_A(0)}.\]

\begin{lem}\label{A-noeth}
$A$ is Noetherian.
\end{lem}

\begin{proof}
Since $R$ is Artinian, $\Jac(R)$ is nilpotent whence the standard filtration $I_k$ is   separated
and thus stabilises at $0.$ It follows that   $\gr_{I_\bullet} R$  is finitely generated over the
Noetherian ring $R/I_1$ and thus   is Noetherian itself. The claim follows from \cite[lem.\
1.5]{sch-ven}.
\end{proof}

Now Small's theorem (\cite[cor.\   4.1.4]{mc-rob})
combined with Lemma \ref{A-noeth} tells us that \be \label{small}
\C_A(0)=\C_A(\N(A)),\ee where, for a ring $C,$  we write $\N(C)$ for
its prime radical, i.e.\ the intersection of all prime ideals of
$C.$ In particular, it contains all nilpotent ideals of $C.$ Since
$\Jac(R)$ is nilpotent as $R$ is Artinian, $\Jac(R)_A$
and $\Jac(R)_B$ are  nilpotent by Lemma \ref{twosided-general} and
thus we obtain the inclusions \be \label{jac-nil} \Jac(R)_A\subseteq
\N(A)   \qquad\textrm{and}\qquad \Jac(R)_B\subseteq
\N(B) .\ee

Next we need the following
\begin{prop}
If $R$ is an (Artinian) semisimple ring, then $B$ is so, too.
\end{prop}

\begin{proof}
We have already seen above that $B$ is Artinian and hence it
suffices to show that $\N(B)=\Jac(B)=0.$ It is a fact (\cite[thm.\
0.2.6]{mc-rob}) that for an arbitrary ring $C$ the radical $\N(C)$
consists precisely of the {\em strongly   nilpotent }
elements of $C,$ i.e.\ those elements $f\in C,$ for which every
sequence $f=f_0,f_1, f_2,\ldots$ with $f_{i+1}\in f_i C f_i$ for all
$i\geq 0,$ is   ultimately  zero. Now let $f\neq 0$ be in
$\N(B).$ After multiplication with   an appropriate
power of $t$ we may and do assume that \[f=b_0 +b_1 t +\ldots \in
A\] with $a:=b_0\neq 0.$ We claim that $a$ is strongly nilpotent. To
this aim let $a_0=a, a_1, a_2,\ldots$ be any sequence of elements in
$A$ with $a_{i+1}=a_i r_i a_i$ for some $r_i\in R.$ We put $f_0:=f$
and $f_{i+1}:=f_i r_i f_i \in f_iRf_i\subseteq f_i Af_i.$ Since $f$
is strongly nilpotent, it follows that $f_m=0$ for $m$ sufficiently
large. But inductively we see that the constant term of $f_m$ is
nothing else than $a_m$   (note that $\delta = 0$ under
the present assumptions)  whence $a$ is strongly nilpotent, i.e.\
belongs to $\N(R)=\Jac(R)=0,$ a contradiction.
\end{proof}

The proposition implies that for $R$ Artinian (but not necessarily
semisimple), $B/\Jac(R)_B\cong (R/\Jac(R))((t,\overline{\sigma}))$
is semisimple, whence $\N(B/\Jac(R)_B)=\Jac(B/\Jac(R)_B)=0$ and thus
$\N(B)\subseteq \Jac(R)_B.$   An argument in  the proof of
\cite[thm.\  4.1.4]{mc-rob}   shows  that $\N(B)\cap
A=\N(A).$ Thus we obtain
\[\N(A)=\N(B)\cap A\subseteq \Jac(R)_B \cap A=\Jac(R)_A,\]
which combined with \eqref{jac-nil} gives \be\label{nil-jac}
\N(A)=\Jac(R)_A   \qquad\textrm{and}\qquad \N(B) =
\Jac(B) = \Jac(R)_B .\ee Taking \eqref{small} into account we have
proven the following

\begin{prop}\label{SregularArt}
$\C_A(\Jac(R)_A)=\C_A(0)$ is an Ore set of $A$ (consisting of regular elements).
\end{prop}

For the rest of this subsection we assume that $R$ is also
semisimple. Then, by Wedderburn theory, $R$ decomposes into a
product \[R=R_1\times \cdots \times R_m\] of   full matrix
rings \[R_i\cong M_{n_i}(D_i)\]   over skew fields
$D_i.$  The $R_i$ are precisely the minimal ideals of $R.$ Thus, any
ring automorphism $\sigma$ of $R$ maps $R_i$ again onto some
$R_{\sigma(i)}$ where we write, by abuse of notation, $\sigma$ also
for the permutation on the set $\{1,\ldots,m\}$ induced in this way
by the automorphism $\sigma$ of $R.$  By taking the products   of
those $R_i$ which belong to the same   $\sigma$-orbit,
it follows   that the pair $(R,\sigma)$ decomposes into a product of
pairs  $(C_j,\sigma_j), 1\leq j\leq \ell,$ each consisting of a ring
$C_j$ with a ring automorphism $\sigma_j$ of $C_j,$ i.e.\
\[R=C_1\times\ldots \times C_\ell, \] and
\[\sigma=\sigma_1\times \ldots \times \sigma_\ell,\] where $\sigma_j$ denotes the restriction of
$\sigma$ to $C_j.$

The proof of the following result is obvious.

\begin{lem}\label{skewpower-product}
Let $(R,\sigma)$ be the product of pairs $(C_j,\sigma_j), 1\leq
j\leq \ell.$ Then there is a canonical isomorphism of rings
\[R[[t;\sigma]]\cong C_1[[t,\sigma_1]]\times \ldots
C_\ell[[t;\sigma_\ell]].\]
\end{lem}

  Because of Lemma \ref{PC-product}   the
crucial case to consider therefore is  the {\em cyclic case,} i.e.\
the situation where $R$ equals some $C_j$ as above, i.e.\ we assume
that $R=R_1\times\ldots \times R_n$ (with $R_i$ simple) and that
$\sigma$ is given by $\tau_i:R_i\to R_{i+1}$ for $i=1,\ldots ,n-1$
and $\tau_n:R_n\to R_1$ in the following way:
\[\sigma(r_1,\ldots,r_n)=   (\tau_n(r_n), \tau_1(r_1),\ldots
,\tau_{n-1}(r_{n-1})).\]   We now define isomorphisms of
rings
\begin{align*}
    \psi : R = R_1\times \ldots \times R_n & \to R_1\times \ldots \times R_1\\
    (r_1,\ldots, r_n) & \mapsto
(r_1,\tau_n\circ\ldots \circ\tau_2(r_2),\ldots, \tau_n(r_n))
\end{align*}
and
\begin{equation*}
    \phi=\phi_0\times \id_{R_1}\times \ldots \times\id_{R_1} :
    R_1\times \ldots \times R_1\to R_1\times \ldots \times R_1
\end{equation*}
where $\phi_0 := \tau_n\circ \tau_{n-1}\circ \ldots \circ \tau_1$ as
well as the cyclic permutation
\begin{align*}
    \pi:R_1\times \ldots \times  R_1 & \to R_1\times \ldots \times
    R_1 \\ (r_1,\ldots,r_n) & \mapsto (r_n,r_1,\ldots ,r_{n-1}).
\end{align*}
Then the following diagram of isomorphisms of rings is clearly
commutative:
\[\xymatrix{
  R=R_1\times \ldots \times R_n \ar[dd]_{\sigma} \ar[r]^{\quad\psi} & R_1\times \ldots \times R_1 \ar[d]^{\phi} \\
        & R_1\times \ldots \times R_1 \ar[d]^{\pi} \\
  R=R_1\times \ldots \times R_n \ar[r]^{\quad\psi} & R_1\times \ldots \times R_1  }\]
Note that $\pi\circ \phi\neq \phi\circ \pi$ if $n\geq2,$ but that
\be\label{phi0} (\pi\circ \phi)^n=\phi_0\times \ldots \times
\phi_0.\ee   The isomorphism $\psi$ induces therefore  a
natural identification of rings
\[R[[t;\sigma]]\cong (R_1\times \ldots \times R_1)[[t;\pi\circ\phi]].\]
Furthermore, we obtain   the  injective homomorphism of
rings
\begin{align*}
    (R_1\times \ldots \times R_1)[[x;(\pi\circ\phi)^n]] & \hookrightarrow
    (R_1\times \ldots \times R_1)[[t;\pi\circ\phi]] \\
    x & \mapsto t^n,
\end{align*}
and the latter ring is free of rank $n$ over the   former
(from the left as well as right). By Lemma \ref{skewpower-product}
and \eqref{phi0} the ring $(R_1\times \ldots \times
R_1)[[x;(\pi\circ\phi)^n]]$ can be identified with the product
$R_1[[x;\phi_0]]\times \ldots \times R_1[[x;\phi_0]]$, and we have
shown the following

\begin{prop}\label{finiteextension}In the cyclic case
$R[[t,\sigma]]$ is a free   (left or right)  module of
finite rank $n$ over   a  subring   isomorphic
to  $R_1[[x;\phi_0]]\times \ldots \times R_1[[x;\phi_0]]$
  where $x$ corresponds to $t^n.$
\end{prop}

By  Proposition \ref{finiteextension} the case where $R$ is a simple
Artinian ring, i.e.\ $R\cong M_n(D)$ for some skew field $D,$ merits
special attention. If $\gamma$ denotes an automorphism of $D$ we
write $M_n(\gamma)$ for the   induced automorphism of
$M_n(D)$ given by applying $\gamma$ to each matrix entry.  For our
purposes the following observation will be crucial.

\begin{prop}\label{inner}
Every automorphism  $\sigma$ of the ring $M_n(D)$ decomposes into
the composite
\[\sigma=Int(C)\circ M_n(\gamma) \]
for some automorphism $\gamma$ of $D$ and some   inner
automorphism $Int(C)$ corresponding to an  invertible matrix $C\in
M_n(D).$
\end{prop}

\begin{proof}
This is an easy consequence of the Isomorphism Theorem in \cite[III.
\S 5]{jac} which we shall explain briefly for the convenience of the
reader: Let $s $ be any ring   automorphism of
$\End_D(V)$ for some finite dimensional $D$-vector space $V.$ Then
there exists an automorphism $\gamma$ of $D$ and a $\gamma$-linear
bijective map $S:V\to V$ such that $s(f)=SfS^{-1}.$  We apply this
to the standard $D$-vector space $V=D^n$ and to the automorphism $s$
which corresponds to $\sigma$ under the identification of
$\End_D(D^n)$ with $M_n(D)$ by using the standard basis
$\{e_1,\ldots ,e_n\}$ of $D^n.$ Consider the $\gamma$-linear
bijective map $\Gamma: D^n\to D^n,(d_1,\ldots ,d_n) \mapsto (\gamma
d_1,\ldots,\gamma d_n)$. Clearly, $S\circ \Gamma^{-1}$ is a
$D$-linear map of $V$, say $F,$ given with respect to the standard
basis by an invertible matrix, say $C.$ Thus we obtain that
\be\label{End} s(f)= F\circ\Gamma \circ f \circ\Gamma^{-1}\circ
F^{-1}\ee for all $f\in \End_D(V).$ Now it is easy to check that
under the identification $\End_D(D^n)=M_n(D)$ the automorphism
$\End_D(V)\to \End_D(V), f\mapsto \Gamma\circ f\circ\Gamma^{-1},$
corresponds to $M_n(\gamma)$ and thus \eqref{End} becomes
\[\sigma(A)= CM_n(\gamma)(A)C^{-1}=Int(C)\circ M_n(\gamma)(A)\] for all $A\in M_n(D).$
\end{proof}

Before we can draw our promised conclusion from this proposition we have to prove the following

\begin{lem}\label{skew-inner}
Let $C$ be an arbitrary ring with ring automorphism $\sigma$ and let
$u$ be a unit in $C$.   If $Int(u)$ denotes the inner
automorphism $c\mapsto ucu^{-1}$ of $C$ then there are canonical
isomorphisms of rings
\begin{align*}
    C[[t;Int(u)\circ \sigma]] & \xrightarrow{\cong} C[[t;\sigma]] \\
           \sum a_n t^n & \mapsto  \sum a_n (ut)^n
\end{align*}
and
\begin{align*}
    C[[t; \sigma\circ Int(u)]] & \xrightarrow{\cong} C[[t;\sigma]]
    \\
    \sum a_n t^n &\mapsto \sum a_n(\sigma(u)t)^n=\sum a_n (tu)^n.
\end{align*}
\end{lem}

\begin{proof}
Both maps are obviously bijective and additive.   The
multiplicativity follows by a straightforward computation based on
the multiplication formula (4) in \cite{sch-ven}. E.g., under the
first map the relation $ta= u\sigma(a)u^{-1}t$ corresponds to
$uta=u\sigma(a)u^{-1}ut=u\sigma(a) t,$ which is equivalent to the
law $ta=  \sigma(a) t$ that holds in $C[[t;\sigma]].$
\end{proof}

\begin{cor}\label{matrix-skew}
 Let $\sigma$ be any automorphism of $M_n(D)$. Then there are     isomorphisms of rings
 \[M_n(D)[[t;\sigma]]\cong M_n(D)[[t;M_n(\gamma)]]\cong M_n(D[[t;\gamma]])\] where $\gamma$ is as
 in Proposition \ref{inner}.
  \end{cor}

\begin{proof}
The first isomorphism follows from the Proposition \ref{inner} combined with Lemma
\ref{skew-inner}. The second  one is easily checked.
\end{proof}

The following result is certainly well-known.

\begin{prop}\label{PLI}
For a skew field $D$ the power series $D[[t;\sigma]]$ form a principal  left (and principal right)
ideal domain. All its left (right) ideals are two-sided and they are precisely the ideals of the
form $D[[t;\sigma]]t^n$ with $n\geq 0.$
\end{prop}

\begin{proof}
For lack of a reference known to us we give a proof using the Weierstrass preparation theorem
\cite[cor.\ 3.2]{ven-weier}, which says that any element in $D[[t;\sigma]]$ can be  written as a
unit times a distinguished polynomial (or vice versa) in $t.$ But note that in this situation a
distinguished polynomial is just a power $t^n$ of $t$ with $n\geq 0.$ Now, if $n$ is the minimal
such exponent among all elements of a given left (or right) ideal $I,$ then $t^n$ certainly
generates $I.$ Since $t$ is regular, it follows that $D$ has no zero divisor, i.e.\ is a domain.
\end{proof}

\begin{cor}
Let $\sigma$ be any automorphism of $M_n(D)$. Then  $M_n(D)[[t;\sigma]]$ is a principal left ideal
and principal right ideal ring.
\end{cor}

\begin{proof}
Combine Corollary \ref{matrix-skew} with \cite[prop.\ 3.4.10]{mc-rob}.
\end{proof}

\begin{rem}\label{free}
In the cyclic case (of order $n$) we have shown that $R[[t;\sigma]]$
is free of finite rank as a   (left or right)
$D[[x;\gamma]]$-module. Here $x=t^n$ and $\gamma$ corresponds via
Proposition \ref{inner} to $\phi_0$ which in turn corresponds to
$\sigma.$
\end{rem}

\begin{prop}\label{C=S}
   Assuming that $R$ is semisimple  we have:
\begin{enumerate}
\item $\C_A(0)=   S_l(R \subseteq A)=S_r(R \subseteq
A) =:S$.
\item The set $S$ is an (left and right) Ore set of $A$ consisting of
regular elements; in particular the localisation $A_S$ of $A$ with
respect to $S$ exists.
\end{enumerate}
\end{prop}

\begin{proof}The first claim will follow from Lemma \ref{regular-cofinite}
below. The second claim   is a consequence either of (i)
and Proposition \ref{SregularArt} or of Proposition \ref{pc-ore} and
Lemma \ref{PC-SOre} below.
\end{proof}

\begin{lem}\label{regular-cofinite}   If $R$ is semisimple then,
for any $f \in A$,  the following implications hold:
\begin{enumerate}
\item If $f $ is  left  (right) regular, then $Af$ $(fA)$ is $R$-cofinite.
\item If $Af$  $(fA)$ is $R$-cofinite, then $f$ is  right (left) regular.
\end{enumerate}
In particular, the following statements are all equivalent: \begin{enumerate} \item[(a)]$f$ is
regular,
\item[(b)]$f$ is left regular, \item[(c)]$f$ is right regular, \item[(d)]  $Af$ is $R$-cofinite, \item[(e)]  $fA$ is $R$-cofinite.
\end{enumerate}
\end{lem}

\begin{proof}
First note that all properties involved in this lemma behave well under finite products
$(R,\sigma)=(R_1,\sigma_1)\times \ldots \times (R_m,\sigma_m),$ e.g.\ $f\in A$ being right regular
is equivalent to $f_i=e_if$ being right regular in $A_i=e_iA$ for all $i=1,\ldots,m.$ Indeed, if
one $f_{i_0}$ were a right zero divisor, say $f_{i_0}g_{i_0}=0$ for some $0\neq g_{i_0}\in A_i$ ,
then $f$ were a right zero divisor, too, as $f\cdot (0,\ldots,0,g_{i_0},0,\ldots,0)=0;$ the other
direction is trivial.
  The compatibility of  cofiniteness   with
products was shown during the proof of Lemma \ref{PC-product}. \\
Thus, in order to prove (i) we may and do assume   that
we are in  the cyclic case. Let $f\in A$ be   left regular, i.e.\
the left $A$-module map $A\to A, a\mapsto af,$ is injective. Thus
its cokernel $A/ Af$ is a finitely generated (left)
$D[[x;\gamma]]$-torsion module; here we use again Remark \ref{free}.
By \cite[lemma 5.7.4]{mc-rob} $A/Af$ has finite length
  over $D[[x;\gamma]]$  and thus, $D$ being the unique
simple $D[[x;\gamma]]$-module, $A/Af$ is finitely generated over
$D.$ To prove (ii) assume that $Af$ is $R$-cofinite and  that $g\in
A$ satisfies $fg=0.$ Multiplication by $g$ on the right induces a
surjective map $A/Af\twoheadrightarrow Ag$ of $A$-modules, which
shows that $Ag$ is finitely generated over $R$ and hence over $D.$
By Lemma \ref{torsion} below it follows that $Ag$ is a
$D[[x;\gamma]]$-torsion module. On the other hand $Ag$ is a
submodule of the free (Remark \ref{free}) and hence torsionfree
$D[[x;\gamma]]$-module $A,$ because $D[[x;\gamma]]$ is
a domain  by Proposition \ref{PLI}. Thus $g=0$ which shows that $f$
is right regular. The rest of the lemma follows by symmetry.
\end{proof}

\begin{lem}
\label{torsion} Let $D$ be a skew field,   $\gamma$ any automorphism of $D$ and $N$ a
$D[[x;\gamma]]$-module which, considered as $D$-module via the natural inclusion $D\hookrightarrow
D[[x;\gamma]],$ is finitely generated. Then $N$ is a torsion $D[[x;\gamma]]$-module.
\end{lem}

\begin{proof}Since the unique simple $D[[x;\gamma]]$-module $D=D[[x;\gamma]]/(x)$  has annihilator ideal $(x),$ it
suffices to  show that $N$ has finite length, say $n,$ as $D[[x;\gamma]]$-module, because then $N$
is annihilated by $(x^n).$ Assume the contrary. Then the Noetherian $D[[x;\gamma]]$-module $N$ is
not Artinian, i.e.\ it exists a descending sequence of $D[[x;\gamma]]$-submodules $N=N_0\supseteqq
N_1 \supseteqq \ldots \supseteqq N_i \supseteqq \ldots$ with $\dim_D(N_i/N_{i+1})>0$ for all $i\geq
0.$ This implies that $\dim_D(N)$ is infinite, a contradiction.
\end{proof}

\begin{lem}\label{PC-SOre}
  For semisimple $R$ the properties $PC_l(R \subseteq A)$ and
$PC_r(R \subseteq A)$ both hold.
\end{lem}

\begin{proof}
  By Lemma \ref{PC-product} we may and do assume that we
are in the cyclic case.
Then $A$ is a free $D[[x,\gamma]]$-module of finite rank by Remark
\ref{free}. Now let $L$ be an $R$-cofinite left ideal of $A,$ i.e.\
$N:=A/L$ is finitely generated over $D$. As seen in the proof of
Lemma \ref{torsion} the module $N$ is annihilated by a power of $x$
and thus by some power of $t.$ In particular, $t^n\cdot 1\in L$ for
some $n\geq 0.$ It follows that $L$ contains the  $R$-cofinite
principal ideal $At^n.$ The right version is shown similarly.
\end{proof}

\subsection{The general case}

We put $\bar{R}:=R/\Jac(R),$   $\bar{A}:=A/\Jac(R)_A,$
$\bar{M}:=M/\Jac(R)M$  for a (left) $R$-module $M$ and we recall the
topological Nakayama lemma from \cite[lem.\ 4.9]{vdbergh}:

\begin{lem}\label{topNakayama}
Let $M$ be a pseudocompact $R$-module. Then $M$ is finitely generated over
$R$ if and only if $\bar{M}$ is   finitely generated over $\bar{R}.$
\end{lem}

\begin{lem}\label{reductionPC}
The property $PC_l(R \subseteq A)$ (respectively $PC_r(R \subseteq
A)$) holds if and only if $PC_l(\bar{R} \subseteq \bar{A})$
($PC_r(\bar{R} \subseteq \bar{A})$) holds.
\end{lem}

\begin{proof}
Note that for $f\in A$ the $R$-module $Af$ is pseudocompact being the epimorphic image of the
pseudocompact $R$-module $A$ under the continuous $R$-linear map $A\to A, a\mapsto af,$ thus also
the quotient $A/Af$ is a natural pseudocompact $R$-module. Hence it follows from the topological
Nakayama lemma \ref{topNakayama} that \be \label{Nakayama} A/Af \mbox{ is finitely generated over
}R \Leftrightarrow \bar{A}/\bar{A}\bar{f} \mbox{ is finitely generated over }\bar{R}.\ee Here
$\bar{f}$ denotes the image  of $f$ in $\bar{A}.$ Now assume that $PC_l(\bar{R} \subseteq \bar{A})$
holds and let $L$ be a $R$-cofinite left ideal of $A.$ Then $A/\Jac(R)+L$ is finitely generated
over $\bar{R}$ and thus there exists an element $f\in L$ such that $\bar{A}/\bar{A}\bar{f}$  is
finitely generated over $\bar{R}$ whence $Af$ is $R$-cofinite by the above equivalence. The
opposite implication is trivial and the right version follows as usual by symmetry.
\end{proof}

\begin{prop}
$\C_A(\Jac(R)_A)$ consists of regular elements.
\end{prop}

\begin{proof}
Let $a$ be in $\C_A(\Jac(R)_A)$ and assume that $ab=0$ for some
$b\in A.$ Then we have also $\bar{a}\bar{b}=0$ for the images
$\bar{a}$ and $\bar{b}$ of $a$ and $b$ in
$A_n:=A/\Jac(R)^n_A\supseteqq R_n:=R/\Jac(R)^n.$ Since by
Proposition \ref{SregularArt} $\bar{a}\in\C_{A_n}(
\Jac(R_n)_{A_n})=\C_{A_n}(0)$ is regular, $b$ must belong to
$\Jac(R)^n_A$ for all $n\geq0.$ But it is easily seen from the
definition of $(-)_A$ and the fact that the Noetherian pseudocompact
ring $R$ is Hausdorff with respect to the $\Jac(R)$-adic topology
(cf.\ \cite[rem.\ 0.1 i]{sch-ven}) that
\[\bigcap \Jac(R)^n_A\subseteq \Big(\bigcap\Jac(R)^n\Big)_A=0. \] Thus $b=0$ and $a$ is right regular. By symmetry we obtain also left regularity.
\end{proof}

\begin{thm}\label{main-Ore}
\begin{enumerate}
\item $\C_A(\Jac(R)_A)=S_l(R \subseteq A)=S_r(R \subseteq A)=:S$.
\item The set $S$ is an (left and right) Ore set of $A$ consisting of
regular elements; in particular the localisation $A_S$ of $A$ with
respect to $S$ exists.
\end{enumerate}
\end{thm}

\begin{proof} We set as before $\bar{R}=R/\Jac(R)$ and
$\bar{A}=A/\Jac(R)_A\cong \bar{R}[[t;\bar{\sigma}]]$ with
$\bar{\sigma}$ the automorphism of $\bar{R}$ induced by $\sigma.$ By
Proposition \ref{C=S} we know that $\C_{\bar{A}}(0)=S_l(\bar{R}
\subseteq \bar{A})=S_r(\bar{R} \subseteq \bar{A})=:\bar{S}.$ Since
$\C_A(\Jac(R)_A)$ is the full preimage of $\C_{\bar{A}}(0)$ and
since, by \eqref{Nakayama}, $S_l( {R} \subseteq {A})$ and $ S_r( {R}
\subseteq {A})$ are the full preimages of $S_l(\bar{R} \subseteq
\bar{A})$ and $S_r(\bar{R} \subseteq \bar{A}),$ respectively, under
the canonical projection $A\twoheadrightarrow \bar{A}$ the first
claim follows and $S$ is the full preimage of $\bar{S}.$ Now $S$ is
an Ore set since the conditions  $PC_l(R \subseteq A)$   and $PC_r(R
\subseteq A)$ are satisfied by Lemmata \ref{reductionPC}
  and  \ref{PC-SOre}.
\end{proof}

  For the sake of completeness we explicitly state the
following

\begin{prop}
 \begin{enumerate}
   \item $\Jac(B) = \Jac(R)_B$.
   \item $B/\Jac(B)$ is the quotient ring of $A/\Jac(R)_A$.
   \item $\Jac(A_S)=\Jac(R)A_S$.
 \end{enumerate}
\end{prop}
\begin{proof}
As a consequence of \eqref{nil-jac} the ideal $\Jac(R)_B$ is equal
to the intersection of all open maximal left ideals in the
pseudocompact ring $B$. But according to \cite[IV.4 prop.\
13.b]{gabriel} this intersection in fact coincides with the ordinary
Jacobson radical $\Jac(B)$. This establishes the first claim. In
view of (i) the second claim was discussed already at the beginning
of the previous subsection. For the third claim we first note that
since $S = \C_A(\Jac(R)_A)$ consists of regular elements we have $1
+ \Jac(R)_A A_S \subseteq A_S^\times$ and hence that
\[ \Jac(R)_A A_S = \Jac(R)A_S \subseteq \Jac(A_S). \] On the other
hand we claim that the factor ring $A/\Jac(R)_A$ is semiprime. In
this situation Goldie's theorem then implies that $A_S/\Jac(R)A_S$
is semisimple which gives the reverse inclusion
\[ \Jac(A_S) \subseteq \Jac(R)A_S. \] To establish this claim we use
the isomorphism $A/\Jac(R)_A \cong (R/\Jac(R))[[t;\bar{\sigma}]]$. Simplifying notation we
therefore will show that $R[[t;\sigma]]$ is semiprime for any semisimple $R$. By Lemma
\ref{skewpower-product} we may assume that with the pair $(R,\sigma)$ we are in the cyclic case. In
this situation $R[[t;\sigma]]$ is, by Proposition \ref{finiteextension}, a finite normalising
extension of a ring of the form $R_1[[x;\phi_0]] \times\ldots\times R_1[[x;\phi_0]]$ where $R_1$ is
simple Artinian. This reduces us, by \cite[thm.\ 10.2.4]{mc-rob}, to the case that $R$ is simple
Artinian. Corollary \ref{matrix-skew} together with \cite[prop.\ 3.5.10]{mc-rob} further reduce us
to the case that $R = D$ is a skew field. But according to Proposition \ref{PLI} the ring
$D[[t;\sigma]]$ is prime.
\end{proof}

Now let $I_\bullet$ be a filtration of $R$ satisfying (I). Then the localisation $A_S$ of $A$ with
respect to $S$   is naturally endowed with the filtration $I_\bullet A_S=A_S I_\bullet.$   We
denote the corresponding completion simply by $(A_S)^\wedge$.  We shall prove the following

\begin{prop}\label{BasCompletion}
 $B\cong (A_S)^\wedge.$  Furthermore, for every $k\geq 0$ there are canonical isomorphisms
 $(A/I_kA)_S\cong B/I_kB\cong (R/I_k)((t;\overline{\sigma},\overline{\delta})).$
\end{prop}

It is remarkable that the completion of $A_S$ arises from $A$ just
by `inverting one element: $t.$'

\begin{proof}
The   statement follows  from \eqref{topiso} and the following
isomorphisms  \[B_k\cong (A_k)_{\C_{A_k}(\Jac(R_k)_{A_k})}\cong
A_{\C_A(\Jac(R)_A)}/I_k A_{\C_A(\Jac(R)_A)},\] where we put
$R_k:=R/I_k,$ $A_k:=R_k[[t;\bar{\sigma},\bar{\delta}]]$ and
$B_k:=R_k((t;\bar{\sigma},\bar{\delta}))=(A_k)_T.$
\end{proof}

Since $A \subseteq B$ the natural map $A_S \longrightarrow B$ given
by the above proposition must be injective. In particular, the
filtration $I_\bullet A_S$ is separated.

\begin{prop}
Under the assumptions of Proposition \ref{skew-grad} $B$ is a flat $A_S$-module.
\end{prop}

\begin{proof}
By Proposition \ref{skew-grad} $B$ is a flat $A$-module.
 But on $A_S$-modules we have the natural isomorphism of
functors $B\otimes_{A_S}- = B\otimes_A -.$
\end{proof}

\section{Iwasawa algebras}

We fix a prime $p,$   let $G$  be a compact $p$-adic Lie group, $\O$ the ring of
integers of any fixed finite extension of $\qp$,  and $\kappa$ its residue field. We write $
\La(G)$ for the Iwasawa algebra $ \La(G)$ of $G,$ i.e.\ the completed group algebra of $G,$ with
coefficients $\O,$ while $\Omega(G)$ denotes the completed group algebra of $G$ with coefficients
in $\kappa;$ both rings are well known to be Noetherian. Henceforth we assume that $G$ has a closed
normal subgroup $H$ such that $G/H$ is isomorphic to $\zp.$ Recall from \cite{sch-ven} that
  \[A:=\La(G) \]  is isomorphic to a skew power
series ring $R[[t;\sigma,\delta]]$ over the Iwasawa algebra $R :=
\La(H)$ of $H.$ For this one picks once and for all a topological
generator $\gamma$ of a subgroup of $G$ which maps isomorphically
onto $G/H \cong \zp$ and one defines $t := \gamma - 1$, $\sigma(r)
:= \gamma r\gamma^{-1}$ for $r \in R$, and $\delta := \sigma - \id$.
As a consequence of \cite[lem.\ 1.6]{sch-ven} the
$\sigma$-derivation $\delta$ is topologically nilpotent and hence
$\sigma$-nilpotent. In particular, for any $k \geq 1$ we find an $m
\geq 1$ such that $\delta^m(R) \subseteq \Jac(R)^k$. Clearly the
ideals $\Jac(R)^k$ for $k \geq 1$ are $\sigma$-, $\sigma'$-, and
$\delta$-stable. Hence the assumption (I) holds and the topological
ring
\begin{equation*}
    B := R\hspace{-3 pt}\ll\hspace{-2 pt} t;\sigma,\delta]]
\end{equation*}
exists by Proposition \ref{skewLaurentinfinite}.

The literature we refer to here and later in this section usually assumes $\O=\zp.$ But in every case it is easily checked that the cited results also hold in our slightly more
general situation, for the statements in \cite{cfksv} this is already  remarked on the bottom of
page 203 (loc.\ cit.).

 \begin{rem}\label{t-independent}
  The ring $B$ is independent, up to natural isomorphism, of the
  choice of the element $\gamma$.
\end{rem}

\begin{proof}
Since this also follows indirectly from the subsequent results we
only indicate a direct argument. Let $\gamma'$ be a second choice
and put $t' := \gamma'-1$. Then $\gamma' = h\gamma^\epsilon$ for
some $h \in H$ and $\epsilon\in\z_p^\times.$ Hence $t' =
h(t+1)^\epsilon  -1=hut +(h-1)$ with some unit $u\in A^\times.$ One
checks that with $t$ also $t'$ is invertible in $B$ provided the
powers $(h-1)^i$ tend to zero in $R$ with $i \rightarrow \infty$.
But
\begin{equation*}
    h-1 = (\gamma'-1)(\gamma^{-\epsilon}-1) + (\gamma'-1) +
    (\gamma^{-\epsilon}-1).
\end{equation*}
The term $\gamma'-1$ (and likewise $\gamma^{-\epsilon}-1$) lies in
the Jacobson radical of $\Lambda(G')$ for some pro-$p$-subgroup $G'
\subseteq G$. Hence its powers tend to zero.
\end{proof}

We fix a closed normal subgroup $N$ of $G$ which is contained in $H$
as an open subgroup and which is a pro-$p$-group. By $\N(H)$ we
denote the preimage of the prime radical $\N(\Omega(G/N))$ of
$\Omega(G/N)$ under the canonical projection
\[\La(G)\To \Omega(G/N),\]  the kernel of which we denote by
$\m(N).$ The definition of $\N(H)$ is independent of the choice of $N$ by
\cite[Lem.\ 2.5]{cfksv}. Then the set
\[S:=\C_A(\N(H))=\C_A(\m(N))\] is a (left and right) Ore set
  consisting of regular elements  of $A$ by \cite[prop.\
2.6, thm.\ 2.4]{cfksv} or \cite[thm.\ G]{ard-brown}. The localisation $A_S$ of $A$ at $S$ is
semi-local \cite[prop.\ 4.2]{cfksv}.

\begin{lem}\label{ardakov}
\begin{itemize}
  \item[(i)] $\N(H) = \mathrm{Jac}(R)A$; in particular, $S =
  \C_A(\mathrm{Jac}(R)_A)$.
  \item[(ii)] $\mathrm{Jac}(A_S)=\N(H)A_S = \mathrm{Jac}(R)A_S$.
\end{itemize}
\end{lem}
\begin{proof}
(i) This follows from \cite[prop.\ 6.3]{ard} by a simple lifting
argument. (ii) The first identity is a standard fact. We briefly
recall the argument. Since $S$ consists of regular elements a simple
computation shows that $1 + \N(H)A_S \subseteq A_S^\times$. This
implies that $\N(H)_S \subseteq \mathrm{Jac}(A_S)$. On the other
hand since the ideal $\N(H)$ is semiprime Goldie's theorem says that
$A_S/\N(H)A_S = (A/\N(H))_S$ is semisimple. Hence we must have
$\N(H)A_S = \mathrm{Jac}(A_S)$.
\end{proof}

The additional assumption $\delta(R) \subseteq \mathrm{Jac}(R)$
which we needed in the previous section does not seem to be
satisfied in this generality. But we do have the following.

\begin{lem}\label{extra-powerful}
\begin{itemize}
   \item[(i)] If $H$ is a pro-$p$-group then the assumption $(\mathrm{SI}_0)$
   is satisfied for the $\mathrm{Jac}(R)$-adic filtration. In
   particular $B$ is pseudocompact and is isomorphic to the
   $\mathrm{Jac}(R)A_S$-adic completion of $A_S$.
   \item[(ii)] If $G$ is a powerful pro-$p$-group then the assumption $(\mathrm{SI})$
   is satisfied for the $\mathrm{Jac}(R)$-adic filtration. In particular
   $\sigma$ induces the identity on the associated graded ring.
   \item[(iii)] If $H$ is an extra-powerful pro-$p$-group then the associated graded
   ring for the $\mathrm{Jac}(R)$-adic filtration is commutative and finitely
   generated over $\kappa$.
   \item[(iv)] If $G$ is powerful and $H$ is extra-powerful then $B$ is
   Noetherian and flat over $A_S$ and $A$.
\end{itemize}
\end{lem}

\begin{proof}
(i) In this case $\mathrm{Jac}(R)$ is the unique maximal ideal in
$R$ and $R/\mathrm{Jac}(R) = \kappa$. Hence $\sigma$ induces
the identity on $R/\mathrm{Jac}(R)$ which implies that $\delta(R)
\subseteq \mathrm{Jac}(R)$. Therefore Propositions
\ref{skewLaurentinfinite} and \ref{BasCompletion} apply. (Note that
in this situation we have $\N(H) = \m(H) = \mathrm{Jac}(R)A$. It
then follows directly from Theorem \ref{main-Ore} that $S$ is a left
and right Ore set consisting of regular elements.)

(ii) The assumption about $G$ means that $[G,G] \subseteq
G^{p^\epsilon}$ with $\epsilon = 1$ and $= 2$ for $p$ odd and $p =
2$, respectively. We in particular have
\[
[\gamma,H] \subseteq G^{p^\epsilon} \cap H = H^{p^\epsilon}
\]
where the latter identity comes from the fact that $G/H \cong
\mathbb{Z}_p$ is torsionfree. For any $h \in H$ we therefore have
$[\gamma,h] = g^p$ for some $g \in H$. We now compute
\begin{align*}
    \delta(h-1) & = \sigma(h-1) - (h-1) = [\gamma,h] - h \\
    & = (g^p - 1)h = ((1 + (g-1))^p - 1)h \\
    & = p(g-1)h + \big(\sum_{i \geq 2} \binom{p}{i} (g-1)^i \big)h \in
    \mathrm{Jac}(R)^2.
\end{align*}
This shows $\delta(\mathrm{Jac}(R)) \subseteq \mathrm{Jac}(R)^2$ and
hence that $(\mathrm{SI})$ is satisfied and that $\sigma$ induces
the identity on the associated graded ring.

(iii) The quotient $\mathrm{Jac}(R)/\mathrm{Jac}(R)^2$ as an $\kappa$-vector space is generated by
the cosets of an uniformizing element $\pi$ of $\O$ and $h_1 - 1, \ldots, h_r - 1$ where
$h_1,\ldots,h_r$ is a minimal system of topological generators of $H$. The assumption on $H$ means
that $[H,H] \subseteq H^{p^2}$. For any two $h_i,h_j$ we therefore may write $[h_i,h_j] = h^{p^2}$
for some $h \in H$. We compute
\begin{align*}
    & (h_i-1)(h_j-1) - (h_j-1)(h_i-1)  = h_ih_j - h_jh_i = ([h_i,h_j]-1)h_jh_i \\
    & \qquad = (h^{p^2} - 1)h_jh_i = ((1 + (h-1))^{p^2} - 1)h_jh_i \\
    & \qquad = p^2(h-1)h_jh_i + \frac{p^2(p^2-1)}{2}(h-1)^2h_jh_i + \big(\sum_{k \geq 3} \binom{p^2}{k} (h-1)^k \big)h_jh_i \in
    \mathrm{Jac}(R)^3.
\end{align*}
This means that the corresponding cosets commute in the graded ring.

(iv) By (ii) and (iii) the assumptions in Proposition
\ref{skew-grad} are satisfied.
\end{proof}

To establish the same facts about $B$ also for general $G$ and $H$ we use the following descent
technique. Let $G' \subseteq G$ be an open normal subgroup and put $H' := H \cap G'$, $R' :=
\Lambda(H')$, and $A' := \Lambda(G')$. We pick an element $\gamma' \in G'$ which topologically
generates a pro-$p$-subgroup of $G'$ and  whose image in $G'/H' \cong \mathbb{Z}_p$ is a
topological generator, and   (suspending earlier notation from \S 1) we define $t' := \gamma' - 1$,
$\sigma'(r') := \gamma' r' \gamma'^{-1}$ for $r' \in R'$, and $\delta' := \sigma' - \id$. Then $A'
= R'[[t';\sigma',\delta']]$ and we may define $B' := R'\hspace{-3 pt}\ll\hspace{-2 pt}
t';\sigma',\delta']]$. We introduce the Ore set $S' := \C_{A'}(\mathrm{Jac}(R')_{A'})$. As a piece
of general notation we denote in the following by $\widehat{C}$ the $\mathrm{Jac}(C)$-adic
completion of any given ring $C$.

\begin{lem}\label{cofinal}
The $\mathrm{Jac}(A_S)$-adic and  the $\m(N)A_S$-adic filtrations of
$A_S$ are equivalent.
\end{lem}

\begin{proof}
Since $\N(\Omega(G/N))\subseteq \Omega(G/N)$ is nilpotent by
\cite[thm.\ 2.3.7]{mc-rob} we obtain \[ \N( H)^n\subseteq
\m(N)\subseteq \N( H)\] and thus
  (observing \cite[prop.\ 2.1.16(vi)]{mc-rob})
\[\mathrm{Jac}(A_S)^n\subseteq \m(N)A_S\subseteq \mathrm{Jac}(A_S)\] for some $n\in \n.$ The
claim follows.
\end{proof}

\begin{prop}\label{descent}
We have the natural identifications as bimodules:
\begin{itemize}
    \item[(i)] $A_S = A'_{S'} \otimes_{A'} A$.
  \item[(ii)] $\widehat{A_S} = \widehat{A'_{S'}}
  \otimes_{A'} A$.
  \item[(iii)] $B = B' \otimes_{A'} A$.
\end{itemize}
\end{prop}

\begin{proof}
(i) Obviously we have a crossed product representation
\[A = A'* G/G'.\] Since regularity of ring elements is preserved
under ring automorphisms $S'$ is clearly $G$-stable, thus by
\cite[lem.\ 37.7]{pass-cross} $S'$ is also an Ore set in $A$
consisting of regular elements and it holds that
\[A_{S'}=A'_{S'} \otimes_{A'} A\] as bimodules. But by the subsequent Lemma \ref{sujatha}
we have $A'_{S'} = A_S$. This establishes (i).

(ii) We choose our $N$ in such a way that it is contained in $H'$.
Let $\m'(N)$ be the kernel of the projection
$\La(G')\to\Omega(G'/N)$, and note that
\[\m(N)^n=A\m'(N)^n=\m'(N)^nA=\m'(N)^n\otimes_{A'} A\]
for any $n \geq 1$. Using again the subsequent Lemma \ref{sujatha}
we deduce that
 \[
 \m(N)_S^n = \m(N)^n_{S'} = \m'(N)^n_{S'} \otimes_{A'_{S'}} A_{S'}
 = \m'(N)^n_{S'} \otimes_{A'_{S'}} A_S
 \]
and
 \[A_S/\m(N)_S^n = (A'_{S'}/\m'(N)_{S'}^n) \otimes_{A'_{S'}} A_S =
 (A'_{S'}/\m'(N)_{S'}^n) \otimes_{A'} A.
 \]
 By passing to the projective limit with respect to $n$ and using
Lemma \ref{cofinal} we obtain (ii).

  (iii) With $N$ as in (ii) we let $\M(N)$ denote the
maximal ideal of $\La(N).$ We set $R_k:=R/\M(N)^k R,$
$R'_k:=R'/\M(N)^k R'$ as well as
$A_k:=R_k[[t;\bar{\sigma},\bar{\delta}]],$
$A'_k:=R'_k[[t';\bar{\sigma'},\bar{\delta'}]],$  where
$\bar{\sigma},\bar{\delta}$ and $\bar{\sigma'},\bar{\delta'}$ are
induced by $\sigma,\delta$ and $\sigma',\delta'$, respectively.
Finally we put $B_k:=(A_k)_T$ and  $B'_k:=(A'_k)_{T'}$ with
$T=\{1,t, t^2,\ldots \}$ and $T'=\{1,t', t'^2,\ldots \}.$ Note that
we have $\M(N)^kA=\m(N)^k$ and $\M(N)^kA'=\m'(N)^k.$ Using Lemma
\ref{twosided-general} it follows that
\[ A_k = A/\m(N)^k = A/\m'(N)^k A \qquad\textrm{and}\qquad A'_k =
A'/\m'(N)^k \] and hence, as above, that
\[A_k\cong A'_k\otimes_{A'}A\] as bimodules
 and thus \[(A_k)_{T'}\cong (A'_k)_{T'}\otimes_{A'} A=B_k'\otimes_{A'}A\] as
 $B'_{k}-A$-bimodules by Proposition \ref{skewLaurent} (we only know that
 $T'$ is an Ore set of $A'_k$). We
 claim that we have a natural isomorphism \be\label{Bk} (A_k)_{T'}\cong
 (A_k)_{T}=B_k.\ee To this end we first show that $t'$ is invertible in $B_k:$
 Without loss of generality we may and do assume that $\gamma'=h\gamma^s$ for
 some power $s=p^m$ of $p;$ indeed the general setting with $s\in \zp$ will
 follow once we know Theorem \ref{Bthm} (ii) because $\widehat{A_S}$ is totally
 independent of the choice of $t.$ Thus  we have $t'+1=\bar{h} (t+1)^s$ which is
 equivalent to \begin{eqnarray}
 \bar{h}^{-1}t'&=&(t+1)^s-\bar{h}^{-1}\\
\label{tstrich}&=&t^s +\bar{p}bt +1-\bar{h}^{-1}\\
&=&t^s(1-at^{-s})\in B_k
\end{eqnarray}
where $b$ is some appropriate element in $A_k$ and where we set
$a:=-(\bar{p}bt +1-\bar{h}^{-1}).$ By the same argument as in Remark
\ref{t-independent} one sees that $1-\bar{h}^{-1}$ is nilpotent in
$R_k$. Hence $a$ is nilpotent in $B_k$ so that the geometric series
$\sum_{j=0}^\infty (at^{-s})^j\in B_k$ is in fact finite and
provides an inverse of $1-at^{-s},$ which implies that $t'$ is
invertible in $B_k.$ As $A_k$ is a free $A'_k$-module and $t'$ is
regular in $A'_k$ it is also regular in $A_k.$ Hence we obtain a
natural inclusion
\[(A_k)_{T'}\subseteq B_k.\] Next we show that $t^{-1}$ lies already
in $(A_k)_{T'}$ so that $(A_k)_{T'}=B_k$ follows. From
\eqref{tstrich} we obtain the equality
\[z:=(t^{s-1}+\bar{p}b)t=\bar{h}^{-1}t' + (\bar{h}^{-1}-1).\] As in
Remark \ref{t-independent} one sees that $z$ is invertible in
$B'_k\subseteq (A_k)_{T'}.$ Hence $t^{-1}=z^{-1}(t^{s-1}+\bar{p}b)$
belongs to $(A_k)_{T'}.$ Putting everything together we have
established compatible isomorphisms  \[B_k\cong B'_k\otimes_{A'}A\]
of $B'_k$-modules; one easily checks that this isomorphism also
respects the right $A$-module structures. Taking the inverse limit
with respect to the canonical projections and using  Remark
\ref{Bprolimes} we obtain (iii).
\end{proof}

The following Lemma which we have used in the above proof is
essentially due to R.\ Sujatha.

\begin{lem}\label{sujatha}
$A_S=A_{S'}.$
\end{lem}

\begin{proof}
The claim follows from Lemma \ref{PC-functorial} (i): \eqref{sat}
applied to $R' \hookrightarrow A' \hookrightarrow A$ says that all
elements of $S $ are already invertible in $A_{S'};$ thus the latter
is also the localisation of $A$ with respect to $S.$
\end{proof}

\begin{thm}\label{Bthm}
We have:
\begin{itemize}
  \item[(i)] $B$ is pseudocompact and Noetherian.
  \item[(ii)] $B \cong \widehat{A_S}$.
  \item[(iii)] $B$ is flat over $A_S$ and $A$.
\end{itemize}
\end{thm}

\begin{proof}
We choose an open normal subgroup $G' \subseteq G$ which is an
extra-powerful pro-$p$-group (\cite[prop.\ 8.5.2]{wilson}). As
$G'/H'$ is torsionfree $H'$ then is extra-powerful as well. The
analogous assertions of our Theorem for $B'$ and $\widehat{A'_{S'}}$
were already established in Lemma \ref{extra-powerful}. Using
Proposition \ref{descent} the fact that $B$ is Noetherian and is
flat over $A_S$ then follows immediately by base extension.

The topology on $B$, by Lemma \ref{twosided-general} (i), is given
by the filtration $\mathrm{Jac}(R)^nB$. Since the pseudocompact ring
$R$ is finite free over the pseudocompact ring $R'$ the same
topology also is given by the filtration $\mathrm{Jac}(R')^nB$. This
shows that the topology on $B$ coincides with the natural topology
of $B$ as a finite free module (by Proposition \ref{descent} (iii))
over the pseudocompact ring $B'$. Hence $B$ is pseudocompact.

By Proposition \ref{descent} (i) we have a ring homomorphism $B'
\longrightarrow B$ and using Lemma \ref{sujatha} then the
commutative diagram of solid arrows
\begin{equation*}
    \xymatrix{
  A' \ar[d] \ar[r] & A'_{S'} \ar[d] \ar[r] &
  B' \ar[d] \\
  A \ar[r] & A_S = A_{S'} \ar@{-->}[r] & B.   }
\end{equation*}
The dashed arrow then exists by the universal property of
localisation. It is continuous for the $\mathrm{Jac}(A_S)$-adic
topology on $A_S$ since $\mathrm{Jac}(A_S) = \mathrm{Jac}(R)A_S$ by
Lemma \ref{ardakov} (ii). But $B$ is complete. Hence this arrow
extends to a ring homomorphism
\begin{equation*}
    \widehat{A_S} \longrightarrow B.
\end{equation*}
As a homomorphism of bimodules this is, by Proposition
\ref{descent}, of course the base extension to $A$ of the
corresponding homomorphism $ \widehat{A'_{S'}} \longrightarrow B'$.
The latter is an isomorphism by Lemma \ref{extra-powerful} (i).
Hence the former is an isomorphism as well.
\end{proof}

\section{ Overconvergent skew Laurent series }

In this section $R$ denotes a  Noetherian pseudocompact ring. We fix
a ring norm $|\ |$ on $R$ which defines the pseudocompact topology.
To avoid confusion we recall that the function $|\ | : R
\longrightarrow \mathbb{R}_{\geq 0}$ satisfies the axioms
\begin{enumerate}
\item  $|a-b| \leq \max(|a|,|b|)$,
\item  $|a| = 0 \Longleftrightarrow a = 0$,
\item  $|ab| \leq |a||b|$,
\item  $|1| = 1$
\end{enumerate}
for any $a,b \in R$. We also suppose that
\begin{enumerate}
\item[(v)] $|a| \leq 1 \qquad\textrm{for any $a \in R$}$.
\end{enumerate}
As before $\sigma$ is a topological ring automorphism of $R$ and
$\delta$ is a continuous left $\sigma$-derivation. We assume that
\begin{enumerate}
\item[(vi)] $|\sigma(a)| = |a| \qquad\textrm{for any $a \in R$}$.
\end{enumerate}
A standard example for a ring norm on $R$ satisfying (i) -- (vi) is
given by
\[
|a|_\Jac := 2^{-k} \qquad\textrm{if $a \in \Jac(R)^k \setminus
\Jac(R)^{k+1}$}.
\]
We also assume that there is a constant $0 < D <
1$ such that
\begin{enumerate}
\item[(vii)] $|\delta(a)| \leq D |a| \qquad\textrm{for any $a
\in R$}$.
\end{enumerate}
In particular, $\delta$ is $\sigma$-nilpotent. It also follows that
\[
|M_{k,l}(\delta,\sigma)(a)| \leq D^k|a| \qquad\textrm{for any $a \in
R$}.
\]
If $\delta(R) \subseteq \Jac(R)$ and $\delta(\Jac(R)) \subseteq
\Jac(R)^2$ then $|\ |_\Jac$ satisfies this condition (vii) with $D
:= 2^{-1}$.

For any real constant $D < u < 1$ we now introduce the (left)
$R$-submodule
\[ B(|\ |;u) := \{\sum_{i \in \mathbb{Z}} a_it^i \in B | \lim_{i
\rightarrow -\infty} |a_i|u^i = 0\}
\]
of $B$. It carries the norm $|\sum_{i \in \mathbb{Z}} a_it^i|_u :=
\sup_{i \in \mathbb{Z}} |a_i|u^i$. By the proof of Proposition
\ref{skewLaurentinfinite} the formula \ref{mult} gives a
well-defined ``multiplication'' map $B \times B \longrightarrow B$.

\begin{prop}
This multiplication map restricts to a map
\[
B(|\ |;u) \times B(|\ |;u) \longrightarrow B(|\ |;u).
\]
\end{prop}

\begin{proof}
Let $x = \sum_{i \in \mathbb{Z}} a_it^i$ and $y = \sum_{i \in
\mathbb{Z}} b_it^i$ be two arbitrary elements in $B(|\ |;u)$ and put
$xy = \sum_{m \in \mathbb{Z}} c_mt^m$ in $B$ with $c_m = c_m^+ +
c_m^-$ as in \eqref{mult} - \eqref{cm-}. We have
\[
|c_m^+|u^m \leq \max_{j \geq n \geq 0} |a_j| \cdot |b_{m-n}|u^{m-n}
\cdot u^n \leq \max_{n \geq 0} |b_{m-n}|u^{m-n} = \max_{i \leq m}
|b_i|u^i
\]
and
\[
|c_m^-|u^m \leq \max_{n \leq j <0} |a_j|u^j \cdot |b_{m-n}|u^{m-n}
\cdot \big(\frac{D}{u}\big)^{j-n}.
\]
The first inequality obviously implies that $\lim_{m \rightarrow
-\infty} |c_m^+|u^m = 0$. On the other hand, given any constant
$\epsilon > 0$, we find natural numbers $N_0,N_1,N_2 > 0$ such that
for $n \leq j$ we have $|a_j|u^j \cdot |b_{m-n}|u^{m-n} \cdot
\big(\frac{D}{u}\big)^{j-n} \leq$
\[
\left\{
  \begin{array}{lll}
    |x|_u|y|_u \cdot \big(\frac{D}{u}\big)^{j-n} \leq \epsilon & \hbox{for
    $j-n \geq N_1$} & \hbox{(since $\frac{D}{u} < 1$);} \\
    |a_j|u^j \cdot |y|_u \leq
\epsilon & \hbox{for $j \leq -N_0$} & \hbox{(since $\lim_{j
\rightarrow -\infty} |a_j|u^j = 0$);} \\
   |x|_u \cdot |b_{m-n}|u^{m-n}
\leq \epsilon & \hbox{for $m-n \leq -N_2$} & \hbox{(since $\lim_{k
\rightarrow -\infty} |b_k|u^k = 0$).}
  \end{array}
\right.
\]
But $j-n < N_1, j > -N_0, m-n > -N_2$ together imply $m >
-N_0 - N_1 - N_2$. It follows that $|c_m^-|u^m \leq \epsilon$ for $m
\leq -N_0 - N_1 - N_2$.
\end{proof}

\begin{cor}
If the assumption $\mathrm{(I)}$ is satisfied then $B(|\ |;u)$, for
any $D < u < 1$, as well as $B^\dag(|\ |) := \bigcup_u B(|\ |;u)$
are subrings of $B$.
\end{cor}

\begin{rem} (i) In general the ring $B^\dag(|\ |)$  depends on the choice of the norm, but if one restricts to
 `ideal norms' the ring $B^\dag(|\ |)$ is independent of the particular choice of $|\ |.$ More
 precisely let $J_1$ and $J_2$ be two ideals in $R$ such that the $J_i$-adic filtrations define the given
topology on $R$ for $i=1,2.$ Consider the norms \[|a|_i := \rho_i^m \;\;\; \mbox{  if   } \;\;\;a
\in J_i^m \setminus J_i^{m+1},\] where $0 < \rho_i < 1$ are two real constants. Assuming that $|\
|_1$ and $|\ |_2$  also satisfy the conditions (vi) and (vii) above  we have
\[ B^\dag(|\ |_1)=B^\dag(|\ |_2).\]\\
(ii) The pseudocompact topology of $R$ can always be defined by an `ideal norm'.
 \end{rem}

\begin{proof}
In order to prove (i) we write $\rho_2 = \rho_1^\sigma$ for some $\sigma > 0$ and note that there
exist natural numbers $n_i$, $i=1,2,$ such that \be\label{incl} J_1^{n_1} \subseteq J_2  \;\;
\mbox{ and } \;\; J_2^{n_2} \subseteq J_1. \ee holds. We shall show the following inequality
\be\label{inequal} \rho_2^{n_2} |\ |_1^{n_2 \sigma} \leq |\ |_2 \leq  \rho_2^{-1}|\
|_1^{\frac{\sigma}{n_1}},\ee which easily implies that $B^\dag(|\ |_1)=B^\dag(|\ |_2).$

Let $a\in R$ be arbitrary and assume that $|a|_1=\rho_1^n$ and $ |a|_2=\rho_2^m$ for some natural
numbers $n$ and $m.$ Then from \eqref{incl} we obtain $n< n_1(m+1),$ which is equivalent to  $
m>\frac{n}{n_1} -1,$ hence
\[|a|_2=\rho_2^m<
\rho_2^{ \frac{n}{n_1}}\rho_2^{-1}=\rho_1^{\frac{\sigma
n}{n_1}}\rho_2^{-1}=|a|_1^{\frac{\sigma}{n_1}}\rho_2^{-1},
\]
which proves the second inequality, the first one is proven similarly.

For (ii) let $|\ |$ be any ring norm on $R$ (satisfying (i) - (v)) which defines the topology of
$R$. Then $J := \{a \in R : |a| < 1\}$ is an open ideal in $R$. Since $R$ is Noetherian $J$ is
generated by finitely many elements $a_1,\ldots,a_s$. We put
\[
\rho := \max \{ |a| : a \in J\} = \max (|a_1|,\ldots,|a_s|)
\]
and introduce the ideal norm
\[
|a|' := \rho^m \qquad\textrm{if}\ a \in J^m \setminus J^{m+1}.
\]
For a given $0 < \epsilon < 1$ choose an $m \in \mathbb{N}$ such that $\rho^m \leq \epsilon$. Then
\[ J^m \subseteq \{ a \in R : |a| \leq \rho^m \} \subseteq \{ a \in R : |a| \leq \epsilon \}.
\]
It follows that the $J$-adic topology is finer than the given topology. In fact we have
\[
|\ | \leq |\ |'.
\]
On the other hand each $J^n/J^{n+1}$ is finitely generated over the Artinian (and Noetherian) ring
$R/J$. It follows inductively that $R/J^m$ is an $R$-module of finite length. Using \cite[cor.\
3.13]{vG-vB} we see that $J^m$ is open in $R$. Hence the $J$-adic topology coincides with the given
topology of $R$.

\end{proof}




\bibliographystyle{amsplain}

\def\Dbar{\leavevmode\lower.6ex\hbox to 0pt{\hskip-.23ex \accent"16\hss}D}
  \def\cfac#1{\ifmmode\setbox7\hbox{$\accent"5E#1$}\else
  \setbox7\hbox{\accent"5E#1}\penalty 10000\relax\fi\raise 1\ht7
  \hbox{\lower1.15ex\hbox to 1\wd7{\hss\accent"13\hss}}\penalty 10000
  \hskip-1\wd7\penalty 10000\box7}
  \def\cftil#1{\ifmmode\setbox7\hbox{$\accent"5E#1$}\else
  \setbox7\hbox{\accent"5E#1}\penalty 10000\relax\fi\raise 1\ht7
  \hbox{\lower1.15ex\hbox to 1\wd7{\hss\accent"7E\hss}}\penalty 10000
  \hskip-1\wd7\penalty 10000\box7} \def\Dbar{\leavevmode\lower.6ex\hbox to
  0pt{\hskip-.23ex \accent"16\hss}D}
  \def\cfac#1{\ifmmode\setbox7\hbox{$\accent"5E#1$}\else
  \setbox7\hbox{\accent"5E#1}\penalty 10000\relax\fi\raise 1\ht7
  \hbox{\lower1.15ex\hbox to 1\wd7{\hss\accent"13\hss}}\penalty 10000
  \hskip-1\wd7\penalty 10000\box7}
  \def\cftil#1{\ifmmode\setbox7\hbox{$\accent"5E#1$}\else
  \setbox7\hbox{\accent"5E#1}\penalty 10000\relax\fi\raise 1\ht7
  \hbox{\lower1.15ex\hbox to 1\wd7{\hss\accent"7E\hss}}\penalty 10000
  \hskip-1\wd7\penalty 10000\box7}
\providecommand{\bysame}{\leavevmode\hbox
to3em{\hrulefill}\thinspace}
\providecommand{\MR}{\relax\ifhmode\unskip\space\fi MR }
\providecommand{\MRhref}[2]{%
  \href{http://www.ams.org/mathscinet-getitem?mr=#1}{#2}
} \providecommand{\href}[2]{#2}

\end{document}